\documentclass[twoside,11pt]{article}

\usepackage{blindtext}

\usepackage[abbrvbib,preprint]{jmlr2e}

\usepackage{lastpage}

\ShortHeadings{HK Gradient Flows for Entropy Functionals}{HK Gradient Flows for Entropy Functionals}
\firstpageno{1}

\usepackage{bef_jjz} %

\newcommand{\EEE}{\color{black}}            %

\usepackage{preamble}
\usepackage{enumitem} %
\usepackage{cleveref}
\usepackage{tikz}

\renewenvironment{proof}[1][Proof]%
                  {\par\noindent{\bfseries #1.} }{\rule{1ex}{1ex}\\[2mm]}

\numberwithin{equation}{section}
\numberwithin{theorem}{section}
\numberwithin{example}{section}

\renewcommand{\cite}{\citep}

\begin{document}

\title{
    Hellinger-Kantorovich Gradient Flows:
    \\
    Global Exponential
    Decay
    of Entropy Functionals
}

\author{\name{Alexander Mielke} \email{mielke@wias-berlin.de} \\
      \addr
      Humboldt University of Berlin \&
    Weierstrass Institute for Applied Analysis and Stochastics,\\ 
    Berlin, Germany
\AND
\name{Jia-Jie Zhu}\thanks{Corresponding author.} \email{zhu@wias-berlin.de}\\
\addr
Weierstrass Institute for Applied Analysis and Stochastics, Berlin, Germany\\
       }

\maketitle

\begin{abstract}
We investigate a family of gradient flows of positive and probability measures, focusing on the Hellinger--Kantorovich (HK) geometry, which unifies transport mechanism of Otto-Wasserstein, and the birth--death mechanism of Hellinger (or Fisher-Rao).
A central contribution is a complete characterization of global exponential decay behaviors of entropy functionals (e.g. KL, $\chi^2$) under Otto-Wasserstein and Hellinger-type gradient flows. 
In particular, for the more challenging analysis of HK gradient flows on positive measures---where the typical log-Sobolev arguments fail---we develop a specialized shape-mass decomposition that enables new analysis results. 
Our approach also leverages the (Polyak-)\L{}ojasiewicz-type functional inequalities and a careful extension of classical dissipation estimates.
These findings provide a unified and complete theoretical framework for gradient flows and underpin applications in computational algorithms for statistical inference, optimization, and machine learning.
\end{abstract}


\bigskip
\noindent
\textbf{MSC (2020) Subject Classifications.} 
Primary: 49Q22 \quad (Optimal transport methods)\\
Secondary: 35Q49 \quad (Transport equations)

\section{Introduction}
\label{sec:intro}
We adopt a perspective rooted in the series of works from the 1990s that pioneered the study of Otto-Wasserstein gradient flows, as eloquently articulated by \citet{ottoGeometryDissipativeEvolution2001}:
\begin{quotation}
    The merit of the right gradient flow formulation of a dissipative evolution equation is that it separates \emph{energetics} and \emph{kinetics}:
    The energetics endow the state space with a \emph{functional}, the kinetics endow the state space with a (Riemannian) \emph{geometry} via the metric tensor.
  \end{quotation}
In essence, the seminal works such as \citep{ottoGeometryDissipativeEvolution2001,jordan_variational_1998}
enabled a systematic perspective of studying the PDE such as the type
\begin{align*}
    \partial_t \mu = - \DIV \left( \mu \nabla\dFdmu \right) 
\end{align*}
as gradient flows of the energy functional $F$, where $\dFdmu$ is its first variation. The solution $\mu_t$ can be viewed as the dynamics and the solution paths of the measure optimization problem
$\displaystyle\min _{\mu \in \mathcal{P}(\mathbb{R}^{d})} F(\mu)$
in the Wasserstein space
of probability measures with finite second moment,
denoted by $\left(\mathcal{P}(\mathbb{R}^{d}),  W_2\right)$.
This perspective has been instrumental in advancing the theory of computational algorithms for
statistical inference and, more recently, machine learning.

\paragraph*{Statistical sampling and particle approximation}
For example, suppose a statistician wishes to generate samples from a probability distribution $\pi$, whose density is in the form
$\textstyle\pi(x) = \frac{1}{\int \rme^{-V(x)}\dd x} \rme^{-V(x)}$,
where $ V$ is referred to as the potential energy function.
This can be cast in the Bayesian inference framework,
that infers the posterior
distribution $\pi$ of some model parameters.
In such applications, one can rely on the fact that $\pi$ is the invariant distribution of the system
associated with the
Langevin stochastic differential equation
$\mathrm{d}X_t = -{\nabla}V(X_t) \ts \mathrm{d}t + \sqrt{2} \ts \mathrm{d}Z_t$,
where $Z_t$ is the standard Brownian motion.
Then,
by numerically simulating the
SDE,
computational algorithms can be designed
to generate samples that approximate those of $\pi$.
From the PDE perspective, this Langevin SDE describes the same dynamical system as the deterministic drift-diffusion Fokker-Planck PDE
\begin{align}
    \partial_t \mu =-\DIV \left(\mu \nabla \left(V + \log\mu\right) \right) 
    \label{eq:drift-diffusion-PDE-intro}
\end{align}
for probability measure $\mu$, which is the \emph{gradient-flow equation} of the Otto-Wasserstein gradient flow of the  KL divergence as driving energy functional,
$F(\mu) = \mathrm{D}_\mathrm{KL}(\mu | \pi)$.
Then, the rigorous analysis developed in the applied analysis context can be used to study the computational algorithms.

\paragraph*{Variational inference and information geometry}
In practice, the exact posterior distribution $\pi$ is often intractable, and one can resort to approximate \emph{variational inference} methods \citep{jordanIntroductionVariationalMethods1999,wainwrightGraphicalModelsExponential2008,blei_variational_2017}.
Different from the Langevin sampler approach,
this amounts to finding the approximate posterior probability measure
by parameterizing $\mu$ with some parameter $\eta\in E \subset \R^n$,
resulting in the optimization problem
\begin{align}
    \min_{\eta \in E\subset \R^n} \mathrm{D}_\mathrm{KL}(\mu_\eta | \pi) .
    \label{eq:vi}
\end{align}
The parameterized distribution $\mu_\eta$ can be chosen from
certain families of distributions, \eg the family of Gaussian
or its mixtures.
In such cases, an efficient approach is the natural gradient descent \citep{amari1998natural,amari2000methods,khan2018fast,hoffman2013stochastic,khanBayesianLearningRule2023} 
on $\eta$ that respects the geometry of the parameterized probability space.
In practice, the update rule is a Riemannian gradient descent scheme
\begin{align}
    \eta^{k+1}\gets
\argmin_{\eta\in E}
\nabla_\eta  F( \mu_{\eta^k})  (\eta - \eta^k)
+
\frac1{2\tau}
(\eta - \eta^k)^\top 
    \bbG_\FR(\eta^k ) (\eta - \eta^k)
    ,
    \label{eq:natural-gradient}
\end{align}
where $ F=  \mathrm{D}_\mathrm{KL}(\cdot  | \pi)$ in the KL variational inference context and $\nabla_\eta  F( \mu_{\eta})$ its Euclidean gradient with respect to $\eta$.
The matrix 
$\bbG_\FR(\nu)
:=\int \mu_{\nu} (x) \cdot \left(\nabla_\nu \log
  \mu_{\nu}\left(x\right)\right)\left( \nabla_\nu \log
  \mu_{\nu}\left(x\right)\right)^\top \dd x$
is 
referred to as the Fisher information matrix, which can be seen as a Riemannian tensor on 
(the tangent bundle of) $ E \subset \R^n$ and hence induces the Fisher-Rao distance over some family of distributions.
It is closely related to the Hellinger distance, which is a central topic in this paper:
the Fisher-Rao distance can be viewed as a restriction of the Hellinger geodesic distance to the submanifold of cerntain families of distributions, e.g., Gaussian; see Remark~\ref{rem:hellinger-vs-fr} for more details.

\paragraph*{Optimization and mirror descent}
In the optimization literature, there is a class of algorithms that uses
the Bregman divergence as the underlying geometry under the name of mirror
descent, \eg
\citep{beckMirrorDescentNonlinear2003,nemirovskijProblemComplexityMethod1983,%
daiProvableBayesianInference2016}.
If the Bregman divergence is chosen as the KL divergence, this approach is
termed the \emph{entropic mirror descent}, \ie an optimization algorithm solving, at the $k$-th iteration,
\begin{align}
    \nu^{k+1}\gets
\argmin_{\nu\in \Delta^d}
\nabla_\nu  F( {\nu^k})  (\nu - \nu^k)
+
\frac1{\tau}
\rmD_{\mathrm{KL}} (\nu | \nu^k)
.
    \label{eq:bregman}
\end{align}
for $\Delta^d:=\mathcal{P}(\Omega)$ where $\Omega$ is a finite discrete set
of cardinality $d$.
In essence, they can be used to solve the optimization
problem $\displaystyle\min _{\mu \in \mathcal{P}(\mathbb{R}^{d})} F(\mu)$ in practice by (i) either considering
the probability measure on a finite domain \citep{beckMirrorDescentNonlinear2003}, or (ii) by
considering a particle approximation of the measure
$\hat \nu=\sum_{i=1}^N w_i \delta_{x_i}$, where $\delta_{x_i}$ is the Dirac
measure at the particle location $x_i$ and $w_i$ is the weight, see e.g.,
\citep{chizat2022sparse,daiProvableBayesianInference2016}.  Furthermore, the
extension to non-gradient flows, albeit finite-dimensional, has been studied in
the optimization literature, \eg
\citep{wibisonoVariationalPerspectiveAccelerated2016,kricheneAcceleratedMirrorDescent2015}
\EEE

\paragraph*{}
The above applications are deeply connected to the gradient flow theory in
various geometries, \ie not just the Otto-Wasserstein geometry.  Specifically, the
Hellinger type gradient flows, which is the focus of this paper, plays a crucial
role and possesses distinct properties when compared with Otto-Wasserstein.
In this paper, we advance the state-of-the-art analysis of gradient flows over positive and probability measures using tools such as the Polyak-L{}ojasiewicz functional inequality.

\paragraph*{Analysis of Polyak-L{}ojasiewicz functional inequalities}
Historically, the celebrated \BE
theorem~\citep{bakryDiffusionsHypercontractives1985} provides a powerful
strategy for analyzing the convergence of the Otto-Wasserstein gradient flow under
the KL-divergence energy.  The \BE condition implies a key functional
inequality, the Logarithmic-Sobolev inequality (LSI)
\begin{align}
  \bigg\|{\nabla \log \frac{\dd\mu}{\dd\pi}}\bigg\|^2_{L^2(\mu)}\geq 
  c_{\textrm{LSI}}\cdot \mathrm{D}_\mathrm{KL}(\mu|\pi) 
  \ \text{ for some } c_{\textrm{LSI}}>0
  .
  \label{eq:LSI}
  \tag{LSI}
\end{align}
The LSI can be viewed as a special case of the (Polyak-)\L{}ojasiewicz inequality
specialized to the Otto-Wasserstein gradient flow of the KL energy
$\rmD_\mathrm{KL}(\mu|\pi)$.
It provides a powerful tool for characterizing the convergence of the dynamics,
\eg governed by the Langevin SDE.  From the optimization
perspective, this is equivalent to analyzing the optimization dynamics of the
problem
\begin{align*}
  \min_{\mu\in A\subset \mathbf{\mathcal P }} {\rmD_\mathrm{KL}}(\mu | \pi )
\text{  in  the space of  }
(\mathcal P, { W_2}).
\end{align*}
The main goal of this paper is to extend this type of functional inequalities,
and consequently the convergence analysis, to a broader class of gradient flows
beyond the now-standard setting of Otto-Wasserstein flows in $(\calP, W_2)$ and the KL energy functional
$\rmD_\mathrm{KL}$.
We now briefly elaborate on those two aspects of our contribution.

\paragraph*{Generalizing the energy functional: from KL to $\phiP$-divergences}
The KL-divergence is by no means the only entropy-type
divergence that possesses interesting properties.
For example, some works by \citet{chewiSVGDKernelizedWasserstein2020,lu2023birth} also consider the $\chi^2$-divergence as the driving energy for machine learning applications.
\citet{zhu2024inclusive} shows that many existing machine learning algorithms are performing the forward KL (also referred to as the inclusive KL therein) minimization via kernelized Wasserstein gradient flows.
For this reason,
we first generalize the KL-divergence energy in \eqref{eq:vi} to a commonly
used family of divergence functional, the \emph{$\varphi$-divergence}
\citep{csiszar1967information}.  We now define this divergence on non-negative
measures $\Mplus(\Omega)\times \Mplus(\Omega)$ as
\begin{equation}
  \label{eq:phi-p-div}
  \mathrm{D_\varphi} (\mu | \nu):=
  \begin{cases}
    \int
    \varphi (\frac{\dd\mu }{\dd \nu })\dd \nu  & 
    \textrm{if }  \mu \ll \nu \ \text{(i.e.\ $\mu=f\nu$ for some $f\in L^1(\nu)$)}, 
    \\
    +\infty, & \textrm{otherwise}.
  \end{cases}
\end{equation}
where $\varphi:[0,+\infty)\to[0,+\infty]$ is a convex entropy generator
function that
satisfies
\begin{align}
 \varphi(1)=\varphi'(1)=0, \ \varphi''(1)=1.
  \label{eq:phi-ent-props}
\end{align}
We delve specifically into the concrete
instantiations of the \Loj inequality
for the following
power entropy generator functions.
\begin{equation}
    \label{eq:power-ent}
    \phiP(s):=\frac{1}{p(p-1)}\left(s^p-ps+p-1\right),\quad p\in \mathbb{R}\setminus\{0,1\},
    \end{equation}
    which satisfies \eqref{eq:phi-ent-props} and $\DDphiP(s)=s^{p-2}$.
    Using the property $\phiP(s)= s \varphi_{1-p}(1/s)$, we obtain the symmetry 
$D_\phiP(\rho| \pi) = D_{\varphi_{1-p}}(\pi|\rho)$.
    Many commonly used divergences can be recovered using the power entropy, \eg
    \begin{equation}
    \begin{aligned}
      &\begin{array}{ll}
        \textrm{KL:}\  \varphi_1(s):=s\log s-s+1, 
        & 
        \textrm{fwd-KL:}\ \varphi_0(s):=s-1-\log s, 
        \\
        \chi^2:\ \varphi_2(s)=\frac{1}{2}(s-1)^2, 
        & 
        \textrm{rev-}\chi^2:
        \varphi_{-1}(s)=\frac{1}{2}\left(s+ \frac{1}{s} -2\right).
      \end{array}
    \end{aligned}
\end{equation}
We refer to the resulting
divergence functional $\mathrm{D_{\phiP}}$ as the
\emph{$\phiP$-divergence} or the $p-$relative entropy (cf. \citep{ohtaDisplacementConvexityGeneralized2011}).
Slightly abusing the terminology due to a scaling factor, we refer to
the power-like entropy generated by $\varphi_{\frac12}$ as the squared Hellinger distance,
\ie
\begin{align}
    \label{eq:fr-power-ent-scaling-12}
    \frac12\He^2(\mu,\nu)
    =
    \int \varphi_{\frac12}\left(\frac{\delta \mu}{\delta \nu}\right) \dd \nu.
\end{align}
We plot the corresponding entropy generator functions in
Figure~\ref{fig:power-ents}.
\begin{figure}[htb]
    \centering
    \includegraphics[width=0.7\linewidth,trim=0 23 0 0, clip=true]{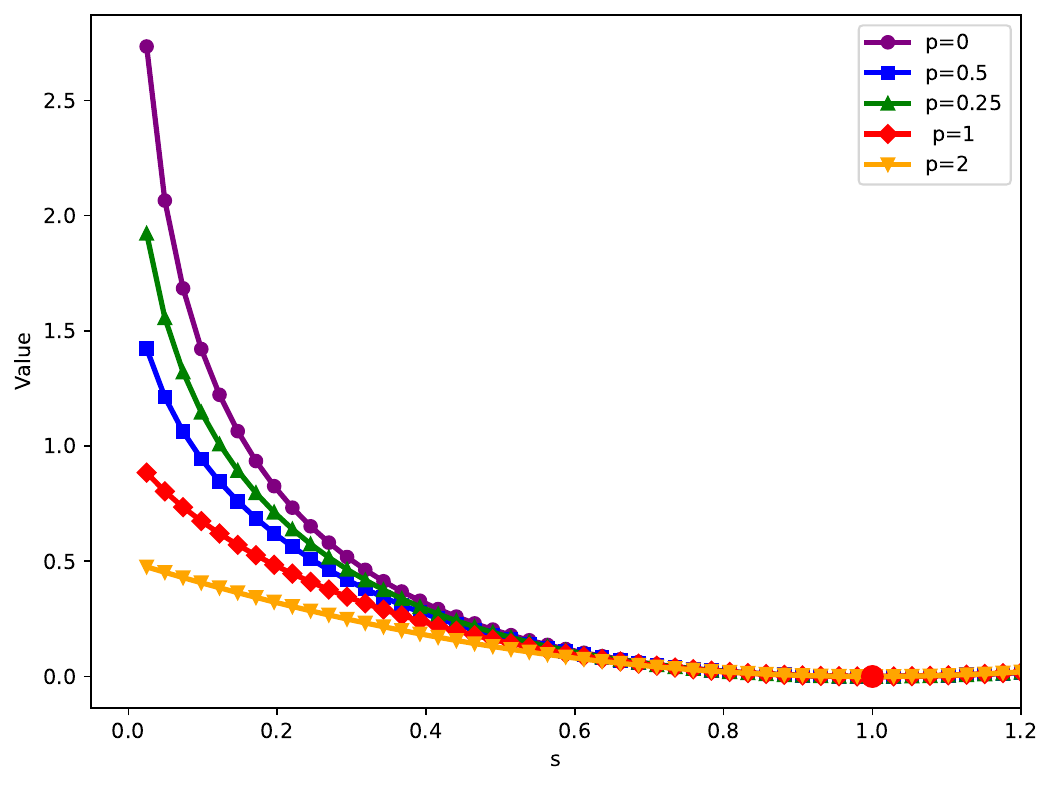}
    \caption{
    The plot illustrates the power-like entropy generator functions $\phiP(s)$ for $s\in [0,1.2]$ and different $p$:
    purple $p=0$ (forward KL), 
    green $p=0.25$,
    blue $p=0.5$ (Hellinger),
    red $p=1$ (KL),
    orange $p=2$ ($\chi^2$).
The large red dot represents the equilibrium at $s=1$ where $\phiP(1)=\phiP'(1)=0$.
}
    \label{fig:power-ents}
\end{figure}
Alternatively, one may use the Hellinger integrals to define the
$\alpha$-divergence $\mathrm{D_\alpha} (\mu | \nu):=
  \frac4{1-\alpha^2}
(1- \int \mu^{\frac{1+\alpha}{2}}  \nu^{\frac{1-\alpha}{2}} )$ with
$\alpha \in (-1,1)$, \EEE
from which one also obtains the KL, forward KL, and the Hellinger as special
cases (for $\alpha\to 1$, $\alpha\to -1$, and $\alpha=0$,
respectively). \EEE  

\begin{table}[htb]
    \centering
    \begin{tabular}{|c|c|}
        \hline
        Gradient-flow geometry & Geod. convexity of ${\phiP}$-divergence\\
        \hline
        Otto-Wasserstein (\BE) & $p\in [1,2]$ and \eqref{eq:BE-cond} with $
        c>0$ $\implies$
        geod. $c$-cvx
         \\
        \hline
        Otto-Wasserstein (McCann cond.) & $p\in [\frac{d{-}1}{d},\infty) \implies$ geod. cvx\\
        \hline
        HK \citep{LiMiSa23FPGG} over $\Mplus$  & $p\in [\frac{d}{d+2},\frac12]\cup (1,\infty) \implies$ geod. cvx
        \\
        \hline
    \end{tabular}
    \caption{
    \EEE    
    Geodesic convexity 
    of
    ${\phiP}$-divergence
    }
    \label{tab:power-entropies-convexity}
\end{table}

\paragraph*{Generalizing the flow geometry: from Otto-Wasserstein to Hellinger-Kantorovich}
In addition to generalizing the energy functional, we extend the analysis of
gradient flows beyond the standard Otto-Wasserstein geometry.
Similar to \eqref{eq:LSI} in that case, one can examine the validity of such \Loj
type functional inequalities when considering general energy functional $F$ in
other gradient-flow geometries; see Table~\ref{tab:diff-types-flows}.
\begin{table}[htb]
    \centering
    \begin{tabular}{|c|c|}
        \hline
        Gradient-flow geometry & Specialized \Loj-type inequality ($\alpha, \beta>0$)  \\
        \hline\hline
        Hellinger & $\big\|{\dFdmu}\big\|^2_{L^2_\mu}
        \geq c\cdot \left(F(\mu) - F_{\inf}\right)$  \\
        \hline
        Spherical Hellinger & $\big\|{\dFdmu - \int \dFdmu \dd \mu }\big\|^2_{L^2_\mu}
        \geq c\cdot \left(F(\mu) -F_{\inf}\right)$  \\
        \hline
        Hellinger-Kantorovich (WFR)
        & 
        $   
        \alpha
        \big\|\nabla \dFdmu\big\|^2_{L^2_{\mu}}
        +
        \beta
        \big\|\dFdmu\big\|^2_{L^2_{\mu}}
        \geq c\cdot \left(F(\mu) - F_{\inf}\right)$  
        \\
        \hline
        Spher. Hellinger-Kantorovich
         & $\!
         \alpha \big\|\nabla \dFdmu\big\|^2_{L^2_\mu}
         {+}
         \beta \big\|{   \dFdmu {-} \!\int \!\dFdmu \!\dd \mu }\big\|^2_{L^2_\mu}
        \geq c\!\; \left(F(\mu) {-} F_{\inf}\right)$  \\
        \hline
    \end{tabular}
    \caption{\Loj inequalities for different gradient flows, where $F_{\inf}:=
      \inf_\mu F(\mu)$. } 
    \label{tab:diff-types-flows}
\end{table}
Our main topics of study are
gradient flows in the Hellinger-Kantorovich (HK) geometry, which
independently
discovered by a few groups of researchers~\citep{chizatInterpolatingDistanceOptimal2018,chizat_unbalanced_2019,liero_optimal_2018,kondratyevNewOptimalTransport2016,gallouet2017jko}.
It is an infimal convolution (inf-convolution) of the Hellinger and Wasserstein distances over positive measures $\Mplus$.
Intuitively, gradient flows in the HK and spherical HK (SHK) geometry combine the dissipation mechanisms of the Otto-Wasserstein flow, \ie the transport of mass, and the Hellinger flow, \ie the creation-destruction or birth-death of mass.
It is often referred to as the unbalanced transport geometry and
possesses a richer structure and more advantageous properties than either of the pure flows alone, and is the focus of this paper.
At the same time, the analysis of the HK and SHK gradient flows is more involved than the pure Hellinger or the Otto-Wasserstein gradient flows.

On one hand, the Otto-Wasserstein geometry describes the transport dynamics that can
easily handle the change of support of measures.  However, it suffers from slow
asymptotic convergence in practical applications.
For example, the behavior of its gradient flow of the KL divergence
depends crucially on the log-Sobolev constant.
The reason is that, in the
(quadratic) Otto-Wasserstein setting, the transport over large distances (e.g., of
outliers in a point cloud) has an over-proportional cost. In contrast,
Hellinger type gradient flows enjoy fast asymptotic convergence because
mass can be destroyed and created at other places instead without any
transport.
\begin{figure}[htb]
\centering
\begin{tikzpicture}[scale=1.8]
\draw[very thick, domain =-1.7:1.85] plot (\x, {(1-\x*\x)^2/3-\x/4});
\fill (-1.515,1.0) circle (0.13);
\fill (1.69,0.8) circle (0.07);
\draw [thick, fill=gray!30] (-1.315,0.5) circle (0.13);
\draw[thick,fill =gray!30] (1.41,0.0) circle (0.075);
\draw[ ultra thick , ->] (-1.2,1.0)--(-1,0.5);
\draw[ ultra thick , ->] (1.45,0.8)--(1.15,0);
\begin{scope}[shift={(4,0)}] 
\draw[very thick, domain =-1.7:1.85] plot (\x, {(1-\x*\x)^2/3-\x/4});
\draw [color=gray, thick, fill=gray!30] (-1.315,0.5) circle (0.11);
\fill (-1.315,0.5) circle (0.05);
\fill (1.41,0.0) circle (0.14);
\draw[thick,fill =gray!30] (1.41,0.0) circle (0.07);
\draw[ dashed, ultra thick, ->] (-1.1,0.6) .. controls(0,1) .. (1.15,0.1);
\end{scope}
\end{tikzpicture}

\caption{
        The two figures illustrate the conceptual advantage of combining
         the Otto-Wasserstein and the Hellinger gradient flows. On the left, the particles are transported by
         the gradient descent enabled by the Otto-Wasserstein gradient flow, where masses do not change. On the right, the
         dashed arrow represent the ``teleportation'' of mass enabled by the Hellinger gradient flow, where the positions do not change.
         The size of the dots represents the amount of mass of the particles.
    }
    \label{fig:asym-double-well}
\end{figure}
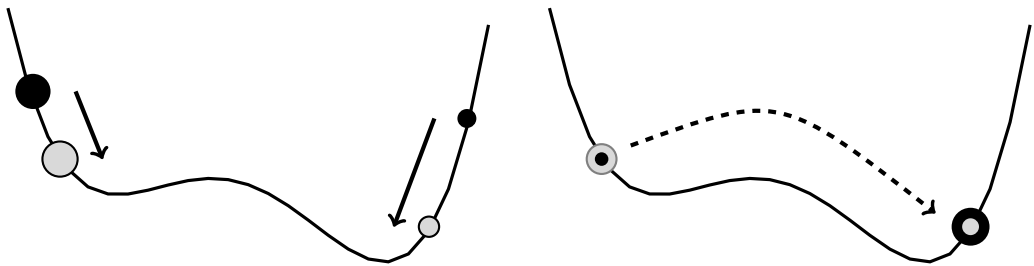
To understand the distinct nature of the two gradient flows, consider an intuitive example of particle gradient
descent method where the measure is approximated using two particles, \ie
$\mu = w_1\delta_{x_1} + w_2\delta_{x_2}$ with $ w_1+w_2=1$ and $ w_i>0$.  The objective function is
an asymmetric double-well potential.  The minimization is initialized as solid
dots in the illustration in Figure~\ref{fig:asym-double-well}.  In this case,
it is easy to see that the gradient descent for each individual particle,
induced by the Otto-Wasserstein gradient flow, will get stuck in the local minimum,
as illustrated as dashed dots.  On the other hand, the birth-death process,
induced by the (spherical) Hellinger gradient flow, can easily teleport the
mass to the right well, but it does not allow the change of the location of the
particles.
An intuitive idea is to consider
the \emph{Hellinger-Kantorovich} (HK) gradient flows
to combine the strengths of both the
Hellinger and Otto-Wasserstein geometries, while overcoming their
weaknesses as illustrated in Figure~\ref{fig:asym-double-well}.

 \paragraph{Overview of the main results}
 In making the above intuition precise, 
this paper advances the theory for the HK and SHK gradient
flows by establishing a few new and precise analysis results.
We provide complete and nontrivial answers to the following open
question:
\begin{quote}
  For the commonly used entropy functionals, \eg (reverse) KL divergence, forward KL,
  Hellinger distance, $\chi^2$-divergence, reverse $\chi^2$, what is the
  convergence behavior of gradient flows in geometries such as the
  Hellinger-Kantorovich space of positive measures $(\Mplus, \HK)$, or the
  spherical Hellinger-Kantorovich space of probability measures
  $(\calP, \SHK)$? 
  Similar to the \BE Theorem and
  \eqref{eq:LSI} in the standard setting of
  $(\calP, \rmD_\mathrm{KL}, W_2)$,
  can we establish precise conditions for global convergence of the gradient flows in all the geometries mentioned above?
\end{quote}
A few major results are summarized in
Table~\ref{tab:major-results}.
\begin{table}[htb]
    \label{tab:major-results}
 \centering
  \begin{tabular}{|c|c|}
    \hline
    Gradient-flow geometry & 
    Global exp.\ decay, \& fcn. ineq.\ for ${\varphi_p}$-divergence\\
    \hline\hline
    Otto-Wasserstein on $\mathcal P$\rule[-1.7em]{0em}{3.8em}
    &
    \parbox{0.63\linewidth}{
      \textbullet\ $\Omega\subsetneq\mathbb R^d$ bounded Lipschitz, \EEE
      $p\geq 1{-}\frac1d$ $\implies$ \L{} with $c_*>0$
      \\
      \textbullet\ $\Omega=\mathbb R^d, p\in [1, 2] $ and (BE) 
      $ \implies $ \L{} with $ c_*=2c_{\textrm{BE}}$ 
    }
    \\
    \hline
     \parbox{0.25\linewidth}{\centering
      Otto-Wasserstein on $\Mplus$
      (Prop.~\ref{prop:no-lsi-w2}, \ref{prop:gen-lsi-w2-mass-preserving})}
      \rule[-1.4em]{0em}{3.2em}
    &
    $\nexists \,c_*>0$ for \L{}; see
    \eqref{eq:LSI-Mplus}
    \\
    \hline
    \! Hellinger on $\Mplus$ \rule[-0.6em]{0em}{1.9em}%
    (Prop.~\ref{prop:loj-phi})\!
    & $p\in (-\infty,\frac12] \iff$ \L{} with $ c_*=\frac1{1{-}p}$  
    \\
    \hline
    \parbox{0.25\linewidth}{\centering
      Spherical Hellinger on $\mathcal P$
      (Thm.~\ref{thm:PolyakLoj.Bh})}
    \rule[-1.7em]{0em}{3.8em}
    & 
    $p\in(-\infty,\frac12] \iff $ \L{} with
    $c_*{=} M_p:=\!
    \begin{cases} \frac1{1{-}p}&\!\!\!\text{for\:} p\leq \frac13 \\[0.2em]
      \!\frac{p(7{-}12p)}{1{-} p} &\!\!\!\text{for\:} p \in 
      [\frac13,\frac12]\!\! \end{cases} 
    $
    \\
    \hline
    \parbox{0.25\linewidth}{\centering
       Hellinger-Kantorovich \\  on $\Mplus$
       (Thm.~\ref{th:DecayEquil})
       } 
    \rule[-2.3em]{0em}{4.9em}
    &
    \parbox{0.63\linewidth}{
     \textbullet\ $p\in (-\infty,\frac12]  \implies$ \L{} with $ c_* =
     \frac1{1{-}p}$ 
     \\ 
     \textbullet\ $p>\frac12  \implies  \textrm{there exists no  } \L{}$ with $c_*>0$
     \\ 
     \textbullet\ 
     $p=1$ and (LSI) $\implies$ 
     No \L{}; exp.\ decay is possible
   }
   \\
   \hline
   \parbox{0.28\linewidth}{\centering
         Spherical Hellinger-Kantorovich on $\mathcal P$
         (Thm.~\ref{thm:SHK-Loj})} 
    \rule[-2.6em]{0em}{5.9em}  
    & 
   \parbox{0.63\linewidth}{
    In general, decay
    rate  $c_*=\max\{ \alpha \;\! c_{\text{\L-W}}\, ,\: \beta \!\; M_p \}$ \\ (see Thm.~\ref{thm:SHK-Loj}). Specifically: \\
     \textbullet\    
      $p\in (-\infty,\frac12] \implies $ \L{} with $ c_*= \frac{1+2p}{1-p}$ 
     \\ 
     \textbullet\ $\Omega\subsetneq\mathbb R^d$  bounded Lipschitz,  \EEE
      $p\in (-\infty,1/2] \cup [1-\frac1d,\infty)$
      \\
      \hspace*{6em} $\implies$ \L{} with  $c_*>0$
     \\
     \textbullet\ $\Omega=\mathbb R^d,\ p\in [1, 2] $ and \eqref{eq:Loj-W} $ \implies $
     \L{} with $ c_*=2c_{\text{\L-W}}$  
  }  
  \\
 \hline
\end{tabular}
\caption{
  Summary of results for \Loj inequalities for 
  ${\phiP}$-divergence energy functional $F(\mu) = \int \varphi_p(\dd \mu/\dd \pi) \dd \pi$
  in different dissipation geometries; see \eqref{eq:power-ent} for the definition of $\varphi_p$. The  $p=1$ case, the KL divergence, corresponds to the well-known
  logarithmic Sobolev inequality
  \eqref{eq:LSI}.
  Remarkably, using the (S)He or (S)HK geometry, the dimension restriction $p>1{-}\frac1d$ in the Wasserstein setting can be circumvented.
 }
\end{table}
In addition and more concretely, we first establish analysis results for the
\Loj inequality for the pure Hellinger gradient flows.  As an example, we show that there is no global \Loj inequality for the Hellinger flow
of the KL energy over positive measures $\Mplus$.
The global
\Loj inequality is significantly more nontrivial to establish than the local
version since we need to create enough metric slope for the Hellinger gradient
flow to escape the initial birth from zero mass; see the illustration in
Figure~\ref{fig:LHS-slopes-loj} and Remark~\ref{rem:power-threshold} for
technical details. Our result captures the fundamental nature of the
Hellinger gradient flows in contrast to the
LSI
for the Otto-Wasserstein.
Going deeper, the analysis of the HK and SHK flows of $\varphi$-divergence is more involved.
We systematically
extract explicit conditions for global convergence
of gradient flows under
$\phiP$-divergence energy functional as defined in \eqref{eq:phi-p-div} and \eqref{eq:power-ent}.
Previously,
the geodesic convexity of energy functionals
has been under scrutiny
in the gradient flow literature;
    see a summary of the 
    implication on the $\phiP$-divergence 
    in Table~\ref{tab:power-entropies-convexity}.
    However, when studying the convergence of gradient flows
    geodesic convexity is a sufficient but not necessary condition.
    For this reason, this paper
    establishes
    a few new functional inequalities that are weaker than geodesic convexity,
    but still sufficient to
    guarantee global exponential decay of the energy functional; see Table~\ref{tab:major-results}
    for the precise statements
    and references to the corresponding theorems for convenience.
    In particular, the standard LSI, when considered for positive measures $\Mplus$, does not hold globally  and must be generalized.
    This adds to the difficulty of establishing the global convergence of the HK gradient flows.
    Nonetheless, using a novel shape-mass decomposition analysis technique, we were able to establish global convergence to equilibrium along the HK gradient flow for the KL divergence as driving energy, see Theorem \ref{th:DecayEquil}.

\paragraph{Other related works}
While there are a few works analyzing
gradient flows in the unbalanced transport geometry,
such analysis, while valuable in its own rights,
has not yet been able to capture the true strength of the
(S)HK gradient flows that this paper showcases.
In \cite{lu2023birth}, the focus is the regime
under a uniform lower bound of the initial density ratio
($\dd \mu_0/\dd \pi$); see \citep[Theorem~2.3]{lu2023birth}.
Various type of assumptions on the initial measure also exist in the literature, such as \citep{domingo-enrichExplicitExpansionKullbackLeibler2023,rotskoff2019neuron}.
From this paper's perspective, such characterizations are \emph{local} and in contrast to this paper's \emph{global} analysis.
We also refer to \cite{chizat2022sparse} for a different type of analysis where he shows that the HK gradient flows of certain functionals, under the assumption of a dense initialization condition, converge to states that satisfy a local \Loj inequality.
A few recent works have also applied the spherical Hellinger-Kantorovich gradient flow with the KL divergence energy functional to practical statistical inference problems~\citep{yanLearningGaussianMixtures2023,lambertVariationalInferenceWasserstein2022}.
\citet{gladin2024interaction} considered the sampling problem using the unbalanced transport gradient flow of the so-called maximum-mean discrepancy functional.
They also exploited a \Loj type inequality to establish the convergence.
Furthermore, this gradient flow is later shown by \citet{zhu2024inclusive} to be a kernel approximation to the Hellinger-Kantorovich gradient flow of the forward KL divergence (\ie $\varphi_0$-divergence).
After the initial preprint version of this paper first appeared on the author's website on January 21, 2024 (\url{https://jj-zhu.github.io/file/ZhuMielke24AppKerEntFR.pdf}), the preprint \citep{carrilloFisherRaoGradientFlow2024} 
appeared on \url{arXiv.org} on July 22, 2024.
It contains an insightful but different analysis of the convergence of the pure spherical Hellinger flow (referred to as Fisher-Rao therein);
see the discussion in Section~\ref{sec:SHK}.

    \paragraph*{Organization of the paper}
    In Section~\ref{sec:background}, 
    we provide background on gradient systems and optimal transport, with a focus on the dynamic formulation and geodesics.
    Section~\ref{sec:analysis}
    is dedicated to the analysis of evolutionary behaviors in the gradient systems using the celebrated Polyak-\Loj
    inequalities.
    There, we introduce the standard log-Sobolev inequality for the pure Otto-Wasserstein gradient flow of the KL-divergence energy.
    Then, we provide novel results on the pure Hellinger gradient flows.
    In Section~\ref{sec:SHK}, we study the unbalanced transport gradient flows restricted to the probability measures, \ie
    the spherical Hellinger-Kantorovich gradient flows.
    There, we were able to establish the global exponential decay of a large class of entropy functionals.
    In Section~\ref{sec:HKGF}, we turn to the Hellinger-Kantorovich gradient flows over the positive measures, where the analysis of functional inequalities is more involved.
    Nonetheless, using a novel shape-mass decomposition in Section~\ref{sec:shape-mass-ana},
    we were able to establish the exponential decay when the LSI is not applicable.
    Additional proofs are given in Section~\ref{sec:proof}.

    \paragraph*{Notation}
    We use the notation $\mathcal{P}(\bar{\Omega}), \Mplus(\bar{\Omega})$ to
    denote the space of probability and positive measures on the closure of
    a open set $\Omega\subset\R^d$ with Lipschitz boundary.
    The base space symbol $\Omega$ is often dropped if there is no ambiguity in the context.
    In this paper, the first variation of a functional $F$ at $\mu\in\Mplus$ is defined as a function ${\frac{\delta F}{\delta\mu}[\mu] }$
    \begin{align}
        \frac{\dd }{\dd \epsilon}F(\mu + \epsilon \cdot v) |_{\epsilon=0}
        = \int {{\frac{\delta F}{\delta\mu}[\mu] }}(x) \dd v (x)
        \label{eq:first-var-def}
    \end{align}
    for any perturbation in measure $v$ such that $\mu + \epsilon \cdot v\in \Mplus$. The Fr\'echet (sub-)differential in a Banach space $(X, \|\cdot\|_X)$ is defined as a set in the dual space
    $$
    \partial F(\mu) :=\left\{\xi \in X^{*} \mid {F}(\nu) \geq {F}(\mu)+\langle\xi, \nu{-}\mu\rangle_{X}+o\left(\|\mu{-}\nu\|_{X}\right) \text { for }\nu\to\mu \right\}
    ,
    $$
    where the small-$o$ notation signifies that the term vanishes more rapidly than the term inside the parentheses.
    When $\partial F(\mu)$ is a singleton, \ie $\partial F(\mu) = \{\xi\} \subset X^*$, we simply write $\mathrm{D}F(\mu):=\xi \in X^*$. \EEE For simplicity, we carry out the Fenchel-conjugation calculation in the un-weighted $L^2$ space.
    Replacing that with duality pairing in the weighted $L^2_{\rho}$ space does not alter the results.
    Common acronyms, such as partial differential equation
    (PDE) and
    integration by parts (IBP), are used without further specifications.
    We often omit the time index $t$ to lessen the notational burden, \eg the measure at time $t$, $\mu(t, \cdot)$, is written as $\mu$.
    In formal calculation,
    we often use measures and their density interchangeably,
    \ie$\int f\cdot \mu$ means the integral w.r.t. the measure $\mu$.
    This is based on the
    standard rigorous generalization from flows over continuous measures to discrete measures \citep{ambrosio2008gradient}.

\section{Preliminaries}
\label{sec:background}

\subsection{Gradient-flow systems and geodesics}
\label{sec:grad-flow-geod}

Intuitively, a gradient flow describes a dynamical system that is driven
towards the fastest dissipation of certain energy, through a geometric
structure measuring dissipation. In this work, we restrict ourselves to the case that the
dissipation law is linear and consequently can be given in terms of a (pseudo)
Riemannian metric. Such a system is called a \emph{gradient system}.  For
example, the dynamical system described by an ordinary differential equation in
the Euclidean space, \(\dot u(t) = - \nabla F(u(t)), u(t)\in \mathbb R^d \), is a simple gradient system.

In this paper, we take the perspective of variational modeling and principled
mathematical analysis, i.e., we study the underlying dynamical systems modeled
as gradient systems specified by the underlying space $X$, energy functional
$F$, and the dissipation geometry specified by the \emph{dissipation
  potential} functional $\calR$.  Given a smooth state space $X$, a dissipation
potential is a function on the tangent bundle $\rmT X$, i.e.\
$\calR= \calR(u,\dot u)$, where, for all $u\in X$, the functional
$\calR(u,\cdot)$ is non-negative, convex, lower semi-continuous, and
satisfies $\calR(u,0)=0$, see \cite{mielke2023introduction} for more
details and motivation. We denote by
\begin{align}
\calR^*(u,\xi) = \sup\bigset{\langle \xi, v\rangle - \calR(u,v)}{ v \in \rmT_u X} 
\label{eq:dual_dissipation_potential}
\end{align}
the (partial) Legendre transform of $\calR$ and call it the \emph{dual dissipation potential}.  
Throughout this work, we will only consider the case
that $\calR(u,\cdot)$ is quadratic, i.e.\
\[
\calR(u,\dot u) = \frac12 \langle \bbG(u)\dot u, \dot u\rangle 
\quad \text{or equivalently} \quad 
\calR^*(u,\xi)=\frac12\langle \xi, \bbK(u) \xi\rangle
.
\]

\begin{definition}[Gradient system]
\label{def:GradSystem}
A triple $(X, F, \calR)$ is called a generalized gradient system, if $X$ is a
manifold or a subset of a Banach space, $F:X\to \R$ is a differentiable
functional, and $\calR$ is a dissipation potential. The associated gradient-flow
equation has the primal and dual form
\begin{equation}
  \label{eq:GFE}
0=\rmD_{\dot u}\calR(u,\dot u) + \rmD F(u) \quad \Longleftrightarrow \quad 
 \dot u= \rmD_{\xi} \calR^*\big(u, {-}\rmD F(u)\big). 
\end{equation}
If $\calR$ is quadratic, we simply call $(X,F,\calR)$ a \emph{gradient system}
and obtain the gradient flow equations
     \begin{align}
       \label{eq:GFE-quadratic}
        0=\bbG(u)\dot u + \rmD F(u) \quad \Longleftrightarrow \quad 
        \dot u= - \bbK(u)\rmD F(u). 
    \end{align}
    $\bbG=\bbK^{-1}$ is called the Riemannian tensor, and $\bbK= \bbG^{-1}$ is
    called the Onsager operator.
\end{definition}
Both forms of \eqref{eq:GFE} and \eqref{eq:GFE-quadratic} have their advantages, but we will often use
the form with
$\calR^*$ and $\bbK$, because they have an additive structure in the cases of interest.

Of particular interest to this paper is the gradient flow in the Hellinger
space of positive measures $(\Mplus,\He)$, also called the \emph{Hellinger-Kakutani} space (cf.\ 
\citep{kakutaniEquivalenceInfiniteProduct1948, 
liero_optimal_2018, LasMie19GPCA}), which is the gradient system that
generates the following reaction equation as its \emph{gradient flow equation}
in the primal form of \eqref{eq:GFE},
\begin{align}
  \partial_t\mu = - \mu \cdot \dFdmu.
\end{align}
This ODE is a consequence of the Hellinger dissipation geometry detailed in
Example~\ref{ex:Fisher-Rao geodesics}.
Alternatively, one can also view the whole right-hand side as the Hellinger
metric gradient induced by the Onsager operator $\bbK_{\He}(\mu)\xi =
\mu\cdot \xi$. 

The Hellinger gradient system is a special case of general gradient flows in
metric spaces, which has gained significant
attention in recent machine learning literature due to the study of the
Otto-Wasserstein 
gradient flow, originated from the seminal works of Otto and colleagues,
e.g.,
\cite{otto1996double, jordan_variational_1998,ottoGeometryDissipativeEvolution2001}.
Rigorous characterizations of general metric gradient systems have been carried
out in PDE literature, for which we refer to \cite{ambrosio2008gradient} for
complete treatments and \cite{santambrogio_optimal_2015,
  peletier_variational_2014,mielke2023introduction} for a first-principles introduction.
To get a concrete intuition, the gradient structure of the
following two classical PDEs will become relevant in later discussions about
Hellinger and Otto-Wasserstein respectively.

\begin{example}
[Classical PDE: Allen-Cahn and Cahn-Hilliard]
Recall the Allen-Cahn PDE
\begin{align}
  \partial_t \rho = \Delta \rho -  V'(\rho),
  \label{eq:allen-cahn}
\end{align}
and the Cahn-Hilliard PDE
\begin{align}
    \partial_t \rho = \Delta \left( -\Delta\rho + V'(\rho) \right)
    .
    \label{eq:cahn-hilliard}
\end{align}
They are the gradient flows of the energy functional
$F(\mu) = \int \big(\frac12 |\nabla \rho|^2 + V(\rho)\big)\dd x $ in two
different Hilbert space geometries, where $V$ is a potential function, e.g.,
the double-well potential 
$V(r) = \frac14 (1{-}r^2)^2$.  The Allen-Cahn equation is the Hilbert-space
gradient-flow equation of the energy $F$ in unweighted $L^2$, \ie $\bbK_\mathrm{AC} = 1$, 
with dissipation potentials
\begin{align}
    \mathcal R_\mathrm{AC} (\mu, \dot u) = \frac12\|\dot u\|_{L^2}^2
  \quad \textrm{and} \quad 
\mathcal R_\mathrm{AC}^* (\mu, \xi) = \frac12\|\xi\|_{L^2}^2
.
\label{eq:ac-grad-geometry}
\end{align} 
Cahn-Hilliard is the gradient flow of $F$ in unweighted $H^{-1}$, \ie
$\bbK_\mathrm{AC} = -\Delta$,
with dissipation potentials
\begin{align}
    \mathcal R_\mathrm{CH} (\rho, \dot u) = \frac12\|\dot u\|_{H^{-1}}^2
, \quad \textrm{and} \quad 
\mathcal R_\mathrm{CH}^* (\rho, \xi) = \frac12\|\nabla\xi \|_{L^2}^2
.
\label{eq:ch-grad-geometry}
	\end{align}
    \label{ex:allen-cahn-hilliard}
\end{example}

\paragraph*{Geodesics and their Hamiltonian formulation.} For many
considerations of gradient flows, the geodesic curves play an important
role. These curves are obtained as minimizers of the length of all curves
connecting two points:
\begin{equation}
\label{eq:GeodArgMin}
\begin{aligned}
\gamma_{\mu_0\to \mu_1}
\in
\argmin_\mu \int_0^1 \langle \bbG(\mu(s))
\dot \mu(s), \dot \mu(s) \rangle \dd s 
\left(= \argmin_\mu \int_0^1 \langle \xi(s), \bbK(\mu(s)) \xi(s) \dd s \right)
\;\\
\Big(\ \ \text{ subject to } \dot \mu(s) = \bbK(\mu(s)) \xi(s)\ \ \Big),
\end{aligned} 
\end{equation}
where $s \mapsto \mu(s)$ has to be absolutely continuous, satisfy
$\mu(0)=\mu_0$ and $\mu(1) = \mu_1$.

In the sense of classical mechanics,
one may consider the dissipation potential
$\calR(\mu,\dot\mu)=\frac12\langle \bbG(\mu)\dot \mu, \dot \mu\rangle $ as a
``Lagrangian'', $L(\mu,\dot \mu)=\calR(\mu,\dot \mu)$, and the dual dissipation
potential $\calR^*(\mu,\xi)= \frac12\langle \xi, \bbK(\mu)\xi\rangle $ as a
``Hamiltonian'', $H(\mu,\xi)=\calR^*(\mu,\xi)$. Then, minimizing the integral
of $L$ is equivalent to solving the Hamiltonian system
\begin{align}
        \Bigg\{
        \begin{aligned}
            \dot\mu &= \partial_\xi H (\mu, \xi)=  \partial_\xi 
              \mathcal R^*(\mu, \xi) = \bbK(\mu) \xi
              \\
            \dot \xi &= -\rmD_\mu H (\mu, \xi)= - \rmD_\mu \mathcal R^*(\mu, \xi)  
        \end{aligned}
        , 
        \quad  \mu(0)=\mu_0, \ \mu(1)=\mu_1.
        \label{eq:intro-gfe-hamilton}
        \tag{H}
\end{align}
Here, the conditions for $u$ at $s=0$ and $s=1$ indicate that finding geodesic
curves leads to solving a two-point boundary value problem.

The theory for geodesics becomes particularly interesting in the case that
$\calR^*$ is affine in the state $\mu$.
Because, then, $\rmD_\mu \calR^*(\mu,\xi)$ no
longer depends on $\mu$ and the system \eqref{eq:intro-gfe-hamilton}
decouples in the sense that the equation for $\xi$ no longer depends on
$\mu$. This particular case occurs in the Otto-Wasserstein, Hellinger, and
consequently Hellinger-Kantorovich (a.k.a. Wasserstein-Fisher-Rao) space.  This
structure allows for the 
derivation of the following characterizations of the geodesic curves and static
formulations of the associated Riemannian distances.  \EEE

\begin{example}
   [Otto-Wasserstein geodesics in Hamiltonian formulation]
\label{ex:Wasserstein geodesics}
In the case of the Otto-Wasserstein geometry, the dual dissipation potential takes the simple form
\begin{align*}
   \calR^*_\Otto(\mu, \xi)
     =  \frac12 \| \nabla \xi \|^2_{L^2_\mu} 
     =\int \frac12 |\nabla \xi|^2 \rmd \mu.
\end{align*}   
The Onsager operator is given by $\bbKotto(\mu)\xi = - \DIV 
(\mu \nabla \xi)$ and the geodesic curves are characterized by
\begin{align}
        \left\{
        \begin{aligned}
            \dot \mu & = - \DIV \left(\mu \nabla \xi \right)
            ,\\
            \dot \xi & = {{-}}   \frac12| \nabla \xi | ^2 .
        \end{aligned}
        \right.
        \label{eq:intro-wgf-hamilton}
        \tag{Geod-W}
    \end{align}
Here, the first equation is the continuity equation which implies that $\mu$ is transported along the vector field $(t,x) \mapsto \nabla \xi(t,x)$, and the second equation is the Hamilton-Jacobi equation, which is notably independent of $\mu$.
The Hopf-Lax formula then gives the explicit characterization of the solution
$$
\xi(s,x) = \inf_{y} \left\{ \xi(0,y) + \frac1{2s} |x{-}y|^2 \right\},
$$
yielding the celebrated dual Kantorovich formulation of the Wasserstein distance.
See \cite{ambrosio2008gradient} for details. 
\end{example}

The main focus of this paper is to study the Hellinger type gradient flows, generated by the geometry of the
Hellinger distance
between two nonnegative
measures $\mu, \nu\in \Mplus$,
\begin{align}
    \label{eq:fr-def}
    \He^2(\mu_0,\mu_1)
    =
    4\cdot \int \left(\sqrt{\frac{\delta \mu_0}{\delta \gamma}}
    - \sqrt{\frac{\delta \mu_1}{\delta \gamma}} \right)^2 \dd \gamma
\end{align}
for a reference measure 
$\gamma\ \textrm{ such that } \mu_0,\mu_1<<\gamma$.
It is straightforward to show that this definition formally coincides with \eqref{eq:fr-power-ent-scaling-12} with the precise scaling factor.
A unique feature of the Hellinger distance~\eqref{eq:fr-def} is that it allows the two measures to have disjoint supports in contrast to divergences such as KL and $\chi^2$.
We recall its dynamic formulation below
using the reaction equation; see also \citep{gallouet2017jko}, \citep{liero_optimal_2018}.
\begin{equation}
\label{eq:bb-formula-fr}
    \begin{aligned}
    \He^2(\mu_0,\mu_1)
    =
    \min
    \left\{\int_0^1
    \| \xi\|^2_{L^2_\mu}
    \dd t
    \ \
    \middle \vert \ \  
    \dot \mu = -  \mu\cdot \xi 
    ,
    \
    \mu(0) = \mu_0,
    \
    \mu(1)= \mu_1\right\}
    .
    \end{aligned}
\end{equation}
If we add
a correction term to the
reaction dynamics, \ie 
$\dot \mu = -  \mu\cdot \left(\xi - \mu\cdot \int \mu  \xi\right)$,
we obtain the spherical Hellinger distance~\citep{LasMie19GPCA} over the probability space $\calP$, instead of positive measures $\Mplus$.
\begin{equation}
  \label{eq:bb-formula-SHK}
      \begin{aligned}
      \SHe^2(\mu_0,\mu_1)
      =
      \min
      \left\{\int_0^1
      \| \xi\|^2_{L^2_\mu}
      \dd t
      \ \
      \middle \vert \ \  
      \dot \mu = -  \mu\cdot\left(\xi - \int \xi\dd \mu\right)
      ,
      \
      \mu(0) = \mu_0,
      \
      \mu(1)= \mu_1\right\}
      .
      \end{aligned}
  \end{equation}
The spherical Hellinger distance, also termed the Bhattacharya distance by 
\citet{Rao45IAAE} 
after its first occurrence in \citep{Bhat42DD}, can be calculated explicitly,
namely 
\begin{equation*}
 \SHe^2(\mu_0,\mu_1) = 4 \arcsin\left( \frac14 \He(\mu_0,\mu_1)\right)
\end{equation*}
see \cite{LasMie19GPCA}, but note the different scaling there. 
Subsequently, we refer to the above as the pure Hellinger and pure spherical Hellinger distances, \ie without the transport aspect of Otto-Wasserstein.

\begin{example}
   [Hellinger geodesics in Hamiltonian formulation]
\label{ex:Fisher-Rao geodesics}
For the Hellinger distance in \eqref{eq:bb-formula-fr}, the primal and dual dissipation potential takes the form 
\begin{equation}
\begin{aligned}
        &
        \mathcal R_\He(\mu,  \dot u)
        = \frac{1}2 
        \left\|\frac{\delta \dot u}{\delta {\mu}}
        \right\|^2_{L^2_{\mu}}
        = \frac12 \int \left| \frac{\rmd \dot u}{\rmd\mu} \right|^2 \dd \mu, 
        \\
        &
         H(\mu, \xi) 
         = \calR^*_{\He}(\mu, \xi)
         =  \frac12  \big\| \xi \big\|^2_{L^2_\mu} 
         =\int \frac12 \xi^2 \rmd \mu,
         \label{eq:fr-hamilton}
\end{aligned}
\end{equation}
where $\frac{\rmd \dot u}{\rmd\mu}$ denotes the Radon-Nikodym derivative
between measures. \EEE 
The Onsager operator is given by $\bbK_\He(\mu)\xi = \xi \mu $ and the geodesic curves are characterized by 
\begin{align}
        \left\{
        \begin{aligned}
            \dot \mu & = - \mu \xi ,\\
            \dot \xi & = {{-}} \frac12| \xi | ^2 .
        \end{aligned}
        \right.
        \label{eq:intro-FR-hamilton}
        \tag{Geod-FR}
\end{align}
Remarkably,
different from the Hamilton-Jacobi setting of Otto-Wasserstein,
this system can be solved in the explicit form  
\[
\xi(s,x) = \frac{\xi(0,x)}{1{+}s \xi(0,{x})/2} \quad \text{and} 
\quad \mu(s,\rmd x) = \big(1{+}s \xi(0,x)/2\big)^2 \mu_0(\rmd x),
\]
where we have used the initial condition $\mu(0)=\mu_0$. 
Applying the terminal condition $\mu(1)=\mu_1$, we arrive at the explicit representation of the Hellinger geodesic
\begin{align}
\gamma_{\mu_0\to\mu_1} \EEE (s) = \big( (1{-}s) \sqrt{\mu_0} + 
 s \sqrt{\mu_1}\big)^2 
 = (1{-}s)^2 \mu_0 + 2s(1{-}s) \sqrt{\mu_0\mu_1} + s^2 \mu_1. 
\label{eq:geod-curve-FR}
\end{align}
see \citep[Eqn.\,(2.8)]{LiMiSa16OTCR} \EEE or \citep{LasMie19GPCA} for details. 
Finally, using
the explicit solution for $\xi(s,x)$ above,
one can show that the Hellinger geodesic distance indeed admits the formula~\eqref{eq:fr-def}.
Formally, one can also obtain
a static dual Kantorovich type formulation
\begin{align*}
    \frac12
    \He^2(\mu_0,\mu_1) =
        \sup_{(2 + \phi)(2 - \psi)=4}
        \left\{
            \int \psi\dd \mu_1
            -
            \int \phi\dd \mu_0 
        \right\}
        .
\end{align*}
\end{example}

\begin{remark}
    [``Hellinger'' versus ``Fisher-Rao'']
    In the literature, the popular naming of ``Fisher-Rao'' has been used
    to describe the infinite-dimensional geometry over probability and
    positive measures. However, the name ``Hellinger'' distance was
    introduced after a paper of Kakutani in 1948 (based on his work 
    \citep{hellingerNeueBegrundungTheorie1909}) and has been used largely since
    the early 1960s, and even by Rao in 1963. We refer to
    \citep[Sec.\,5]{Miel24EVIF} for 
    some historical remarks.  Nevertheless, starting from 2016 some authors such as
    \citet{BaBrMi16UFRM,gallouet2017jko,santambrogioEuclideanMetricWasserstein2017}
    used the name ``Fisher-Rao'' instead, and it is now very popular in imaging
    and machine learning.
    However, many such uses of the name ``Fisher-Rao'' are an abuse of the naming
    convention because it should be used in the sense of Rao's original definition in
\citep{rao1945information} as a way to characterize the distance of measures
within a given submanifold of measures. Thus, the Fisher-Rao distance depends
on the submanifold and is given by the length of the shortest curve within the
submanifold, where length is measured in the Hellinger metric. 

In the present paper, the spherical Hellinger distance $\SHe$ can be understood
as a type of Fisher-Rao distance with respect to the submanifold $\calP(\Omega)$ as a
submanifold of $\Mplus(\Omega)$. Another type of the Fisher-Rao distance occurs, for instance, if one
chooses the submanifold of exponential family distributions. \EEE
    \label{rem:hellinger-vs-fr}
\end{remark}

\subsection{Unbalanced optimal transport: Hellinger-Kantorovich}
\label{sec:SHK-background}
As we have seen in the previous subsection, the Otto-Wasserstein geometry gives us
the transport type dynamics, while the Hellinger geometry provides the
birth-death, also reaction, mass creation or destruction, type dynamics.  A few
groups of researchers, including \citet{chizatInterpolatingDistanceOptimal2018,chizat_unbalanced_2019,liero_optimal_2018,kondratyevNewOptimalTransport2016,gallouet2017jko},
proposed the Hellinger-Kantorovich (HK) geometry, which is the combination of
the Hellinger and Wasserstein distances.  We
refer to their works for the details and
provide below a self-contained
introduction to the HK geometry and gradient flow.

The optimal transport problem of Kantorovich must be generalized for the
transport between measures of different mass to become admissible.
The construction is as follows:
In addition to
the initial and target measures $\mu_0$ and $\mu_1$, one considers measures
$\pi_0$ and $\pi_1$ between which classical optimal transport happens.
Then, the mismatch between $\mu_0$ and $\pi_0$ and between $\mu_1$ and $\pi_1$ is penalized using a divergence functional $\Psi$, \eg the KL divergence. This is then called \emph{unbalanced transport},
defined using the \emph{entropy-transport} functional
\begin{multline*}
  \ET_{c,\Psi} \left(\Pi |\mu_0, \mu_1\right):=
    \Biggl\{
      \int c(x_0, x_1) \dd {\Pi}\left(x_0, x_1\right)
      +  \Psi(\pi_0|\mu_0)
      +  \Psi(\pi_1|\mu_1)
    \bigg|
    \\
    \pi_0(\dd x_0):={\Pi}\left(\dd x_0, \Omega\right)  ,
      \
      \pi_1(\dd x_1):={\Pi}\left(\Omega, \dd x_1\right) 
    \Biggr\}
    ,
\end{multline*}
$c$ is a cost function of transport, \eg the squared Euclidean distance.
In general, functionals
defined using this type of inf-convolution
do \emph{not} generate a (squared) distance
on $\Mplus(\Omega)$. And even if it is a distance, it may not be a geodesic
distance. It was the main achievement of
\citet{LiMiSa16OTCR,liero_optimal_2018} that the HK distance, defined
as a geodesic distance in the sense of the dynamic Benamou-Brenier sense, via
\begin{align}
  \label{eq:bb-formula-HK}
  &\HK^2(\mu_0,\mu_1)    =
  \\ \nonumber
  & \min \left\{ \int_0^1 \!
    {\alpha} \|\nabla \xi\|^2_{L^2_\mu}
    {+} 
    \beta \| \xi\|^2_{L^2_\mu} \dd t
    \ \middle\vert \
    \dot \mu = 
    \alpha \DIV(\mu\cdot  \nabla \xi  )
    {-} \beta \mu  \xi  ,
    \  \mu(0) = \mu_0, \
    \mu(1)= \mu_1\right\},
\end{align}
can be characterized as an unbalanced transport problem as shown below, if $c$ and $\Psi$
are chosen in a very particular way.
Different choices of
$\alpha, \beta>0$ allow us to tune the relative strength of the two
geometries, trading off the transport and the birth-death mechanisms.

\begin{theorem}
[Logarithmic-Entropy-Transport definition of $\HK$]
\label{th:HK.StaticForm}
\emph{\citep[Thm.\,8]{LiMiSa16OTCR}}
The
Hellinger-Kantorovich distance over positive measures $\Mplus$
has
the equivalent characterization as the optimal value of the
\emph{logarithmic-entropy-transport} (LET) problem
\begin{align}
  \HK^2 \left(\mu_0, \mu_1\right):=
  \inf_{
    {\Pi}\in\Mplus(\Omega\times \Omega)
  }
  \ET_{c,\Psi} \left(\Pi |\mu_0, \mu_1\right)
  ,
  \label{eq:HK-def}
\end{align}
where functional $\Psi$ is the (scaled) KL divergence
$\Psi(u|v):= \frac1\beta\,\rmD_{\textrm{KL}}(u|v)$ and the transport cost
is
\[
  c(x_0, x_1):=
  \begin{cases}
    \ds \frac{-2}\beta \, \log
    \left(\cos\left(\sqrt{\frac{\beta}{4\alpha}} \,|x_0{-}x_1|\right)
    \right),
    & \ds\text{for } |x_0{-}x_1| < \pi\sqrt{\frac\alpha\beta},\\
    +\infty, & \text{ otherwise}.
  \end{cases}
\]
\end{theorem}
Intuitively, the HK geometry combines the mechanisms of the Otto-Wasserstein
geometry, \ie the transport of mass, and the Hellinger geometry, \ie the
birth-death of mass.  It possesses a richer structure and more advantageous
properties than either of the pure geometries alone.

The gradient flow in the
HK geometry generates the gradient flow equation, which is the following reaction-diffusion PDE.

\begin{example}
[Reaction-diffusion PDE]
The gradient-flow equation of the HK gradient system over positive measures
$(\Mplus, F, \HK)$ corresponds to
the reaction-diffusion PDE
\begin{align}
    \dot \mu 
    =
    -\alpha \cdot {\bbKotto(\mu)}\ \dFdmu
    -\beta \cdot {\mathbb K_{\He}(\mu)}\ \dFdmu    
    = \alpha  \DIV\left(\mu \nabla\dFdmu\right) -   \beta \mu \dFdmu.
    \label{eq:react-diffus-pde}
\end{align}
\end{example}
The HK geometry and gradient flows are defined over the space of positive measures $\Mplus$.  For many machine learning applications, it is often more
convenient to only work with probability measures.  The restriction of the
HK geometry to the space of probability measures $\calP$ is discussed in 
\cite{LasMie19GPCA}, referred to as the spherical Hellinger-Kantorovich (SHK)
geometry. In this paper's context, we establish the following explicit formula.
\newcommand{\AABB}{{\alpha,\beta}}%
\newcommand{\AAFF}{{\alpha,4}}%
\begin{proposition}[Explicit formula for $\SHK_\AABB$]
For $\alpha,\beta>0$ we have the formula for the spherical Hellinger-Kantorovich distance,
\begin{equation}
  \label{eq:Scaling.SHK}
  \SHK_\AABB(\mu_0,\mu_1) =  \frac4{\sqrt{\beta}}\, 
  \arcsin\Big( \frac{\sqrt{\beta}}4\, \HK_\AABB(\mu_0,\mu_1)\Big).  
\end{equation}
\end{proposition}
\begin{proof}
In \citep{LasMie19GPCA}, the passage from $\HK_\AABB$ to $\SHK_\AABB$ is
discussed in detail by showing how the geodesics of $(\calP(\Omega),\SHK)$ and
$(\Mplus(\Omega),\HK)$ can be transformed into each other. Under the assumption
that $\beta=4$, which is used in the scaling assumption (2.1) and (2.2) therein,
it has been shown that 
\[
\SHK_\AAFF(\mu_0,\mu_1)= 
\arccos
\left(1-\frac12\,\HK_\AAFF(\mu_0,\mu_1)^2\right) =
2 \arcsin\left( \frac12 \HK_\AAFF(\mu_0,\mu_1)\right),
\]
where the first identity follows from \citep[Thm.\,2.2]{LasMie19GPCA} and the
second from the trigonometric identity
$\sin \sigma = \sqrt{\left(1{-}\cos(2\sigma)\right)/2}$. 

It now remains to apply the simple scaling
$\HK_\AABB^2 = \frac4\beta
\HK_{4\alpha/\beta,4}$ and $\SHK_\AABB^2 = \frac4\beta \SHK_{4\alpha/\beta,4}$,
and the assertion follows. 
\end{proof}

Similarly, the spherical Hellinger distance $\SHe=\SHK_{0,1}$, also known as the
Bhattacharya distance, is related to the Hellinger distance by 
\[
\SHe(\rho_0,\rho_1)= 4 \arcsin \left( \frac14\, \He(\rho_0,\rho_1)\right). 
\]
Recall our scaling of $\He$ in \eqref{eq:fr-power-ent-scaling-12} with $\He(0,\mu)=2\mu(\Omega)$,
while some other works use $\wt{\He}=\HK_{0,4}$ giving $\wt{\He}(0,\mu)=\mu(\Omega)$. We also remind the reader of the use of the notation $\rho$ for the probability measure instead of the positive measure $\mu$.

The associated Onsager operator (inverse of the Riemannian metric tensor $\bbG_\SHe$) is given by
restricting that of the Hellinger to the probability measures, namely
\begin{align}
    \bbK_{\SHe}(\rho) \eta = \beta \,\rho \Big( \eta - 
    \int
    \eta \dd\rho \Big).
    \label{eq:SHK-bbK}
\end{align}
Using that relation,
we obtain
the Onsager operator (inverse of the Riemannian metric tensor)
for the spherical Hellinger-Kantorovich (SHK) geometry
\[
\bbK_{\SHK}(\rho)\eta= - \alpha \DIV\!\left(\rho\nabla 
\eta\right) + \beta \rho\left( \eta -{ \int \rho \eta\dd x}\right),
\]
and the SHK gradient flow equation
\begin{align}
    \dot \rho = 
    - \bbK_{\SHK}(\rho)\dFdmu
    =
    \alpha \DIV\!\left(\rho\nabla \dFdmu\right) - \beta \rho\left( \dFdmu -{
    \int
    \rho \dFdmu\dd x}\right).
    \label{eq:SHK-gf}
\end{align}

\section{Functional inequalities: Otto-Wasserstein and Hellinger}
\label{sec:analysis}

Functional inequalities are the building blocks for the analysis of many computational algorithms, such as for sampling and optimization over probability measures.
The main goal of this section is to develop an intuition for the \Loj type inequalities for the Otto-Wasserstein and Hellinger type
gradient-flow geometries.

\subsection{Otto-Wasserstein gradient flow over probability measures $\calP$}
Our starting point is the differential \emph{energy dissipation balance} relation
of gradient flow systems,
\begin{align}
\frac{\dd}{\dd t}F(\mu(t))=
        \iprod{\DF}{\dot \mu}
        =
            -\biggl(\mathcal R(\mu, \dot{\mu})+\mathcal R^{*}(\mu,-\mathrm{D}F)\biggr) 
            =:
            - \calI(\mu(t))
            .
            \label{eq:chain-rule-RRstar}
\end{align}
where the functionals $\calR$ and $\calR^*$ are the primal and dual dissipation
potentials discussed in Section~\ref{sec:grad-flow-geod}.  We refer to the
quantity $\calI$ as the \emph{dissipation} of energy $F$.  It was also referred
to, in some contexts, as entropy production.  The letter $\calI$ is due to
Fisher's information while the letter $\calR$ is due to the Helmholtz-Rayleigh
dissipation principle~\citep{rayleigh_general_1873}.  From this, we introduce
the following version of the \Loj condition.
Note that, in the definition of the functional $\calI$,
it is understood that $ \dot\mu$ is replaced by $\rmD_\xi \calR^*(\mu, - \rmD
F(\mu))$ to obtain a functional of $\mu$ alone. As we are in the quadratic case,
we always have $\calI(\mu)=\calR^*(\mu, - \rmD
F(\mu))$.

\begin{definition}
[Polyak-\Loj inequality for generalized gradient systems]

We say that the Polyak-\Loj inequality holds if
\begin{align}
    \mathcal R(\mu, \dot{\mu})+\mathcal R^{*}(\mu,-\mathrm{D}F)
     =\calI(\mu)  \EEE \geq c\cdot \left(F(\mu(t)) - F_*\right)
    \quad \text{ with } F_*= \inf_\mu F(\mu).
    \label{eq:lojas-cond}
    \tag{\L}
\end{align}
holds for some constant $c>0$.
\label{def:loj-standard}
\end{definition}
\newcommand{\refLojas}{\eqref{eq:lojas-cond}\xspace} For conciseness, this
paper does not analyze more general \Loj inequalities, \ie no higher order
powers on the right-hand side, due to the relevance of \refLojas to computational algorithms in machine
learning and optimization; cf. \citep{karimiLinearConvergenceGradient2020}. We
simply refer to it as the \Loj inequality in the rest of the paper.  We refer
to articles such as
\citep{otto2000generalization,blanchetFamilyFunctionalInequalities2018} for a
wider scope of related inequalities.  An immediate consequence of \refLojas is
that the energy of the gradient system converges exponentially via Gr\"onwall's
lemma,
\ie
$$
\text{\eqref{eq:lojas-cond} } \implies \ 
F(\mu(t)) - F_* \leq \ee^{-c\cdot t} \left(F(\mu(0)) - F_* \right)
.
$$
Therefore, on the formal level, the intuition of the analysis is to produce the 
\Loj type relations in the succinct form of $ \calI \geq c\cdot \left(F -
  F^*\right)$.

Concretely,
in the Otto-Wasserstein gradient flows and the Fokker-Planck PDEs, energy dissipation can be easily calculated
\begin{align}
    \calI(\mu)
    =
    -\frac{\dd}{\dd t}F(\mu)
    \overset{\text{(along WGF)}}{=}
    \int \mu \biggl|\nabla \dFdmu\biggr|^2
    .
    \label{eq:energy-dissipation-w2}
\end{align}
As an already well-known
example, we now formally check the inequality \refLojas for 
the Otto-Wasserstein gradient system with the KL-divergence, \ie $(\mathcal P({\R^d}), \mathrm{D}_\mathrm{KL}(\cdot|\pi), W_2)$, where $\mathrm{D}_\mathrm{KL}$ is defined in \eqref{eq:phi-p-div}, \eqref{eq:power-ent}.
We calculate the dissipation
\begin{align*}
    -\calI(\mu) =
        \frac{\dd}{\dd t}\mathrm{D}_\mathrm{KL}(\mu|\pi) 
        =
        \left\langle{\log \dMudPi }, {-\DIV\left(\mu \nabla \log \dMudPi\right)}
        \right\rangle_{L^2}
        \overset{\textrm{(IBP)}}{=}
        -\biggl\|{\nabla \log \dMudPi}\biggr\|^2_{L^2_\mu}
        .
\end{align*}
Specializing the \Loj inequality \refLojas to this setting, 
we arrive at \EEE the 
\emph{logarithmic Sobolev inequality} (LSI)
\begin{align}
\biggl\|{\nabla \log \dMudPi}\biggr\|^2_{L^2(\mu)}\geq c\cdot \mathrm{D}_\mathrm{KL}(\mu|\pi),
\label{eq:LSI-2}
\tag{LSI}
\end{align}
which needs to hold for some $c>0$. \EEE By Gr\"onwall's lemma, the
entropy decays exponentially, \ie
$\mathrm{D}_\mathrm{KL}(\mu|\pi) \leq \rme^{-c\cdot t}
\mathrm{D}_\mathrm{KL}(\mu(0)\|\pi) $.  \eqref{eq:LSI-2} is a special case of the
(Polyak)-\Loj inequality for the Otto-Wasserstein geometry and the more general
$\varphi$-divergence energies, namely  
\begin{align}
\biggl\|\nabla \varphi^\prime \left(\frac{\dd\mu}{\dd
    \pi}\right)\biggr\|^2_{L^2_{\mu}} 
\geq  c\cdot \operatorname{\mathrm{D}_\varphi}(\mu|\pi) . 
\label{eq:Loj-W}
\tag{\L{}-W}
\end{align}
In particular, 
we will
exhaustively investigate
the $\phiP$-divergence energy functional case.
The inequality~\eqref{eq:Loj-W} reads
\begin{equation}
  \label{eq:LogSob1}
  \frac1{(p{-}1)^2} \int \mu \bigg|\nabla\left(\left(\ddfrac{\mu}{\pi}\right)^{p-1}\right)\bigg|^2 \dd x 
  \geq
  c\cdot \mathrm{D}_\phiP(\mu|\pi) 
\end{equation}
For
$p=1$, \ie 
the choice of $\varphi_{\mathrm{KL}}$ ($\varphi_{1}$-divergence or the $1$-relative entropy),
recovers the \eqref{eq:LSI-2},
which has already been intensely investigated in the literature.
The \BE theorem~\citep{bakryDiffusionsHypercontractives1985} gives a sufficient condition for the logarithmic Sobolev inequality (LSI)
to hold along the solution of the Fokker-Planck equations:
the target probability measure $\pi$ satisfies the \BE condition, if 
$\pi\propto\exp\left(-{V}\right)$
for the potential function $V$ that satisfies
\begin{align}
    \nabla^2V\geq c_\textrm{BE}\cdot \ID,\ 
    c_\textrm{BE}>0
    .
    \label{eq:BE-cond}
    \tag{BE}
\end{align}
Moreover, following \citet{bakryDiffusionsHypercontractives1985},
\citet{arnold2001convex} provided an elementary proof of the \BE theorem
for general $\varphi$-divergence energies that satisfies
\begin{align}
    \varphi(1)=
    \varphi'(1)=0,
    \
    \varphi''(1)>0
    \quad \text{and} \quad
    \left(\varphi^{\prime \prime \prime}(s)\right)^2 \leq \frac{1}{2} \varphi^{\prime \prime}(s) \varphi^{(4)}(s)
    .
    \label{eq:phi-arnold}
\end{align}
Their results state that if
\eqref{eq:phi-arnold} holds,
the Otto-Wasserstein gradient flow with the corresponding $\varphi$-divergence energy converges exponentially.
That is,
the following sufficient relation holds
    \begin{align}
        \eqref{eq:BE-cond} + \eqref{eq:phi-arnold} \implies \eqref{eq:Loj-W}: \textrm{\Loj for Otto-Wasserstein}
        \implies \textrm{exp. decay}.
        \label{eq:be-relation}
    \end{align}
\EEE
First, we 
slightly modify
this result
for the $\phiP$-divergence energy functional
and the case of domain $\Omega=\R^d$.
The proof is straightforward by
    plugging in the definition of the $\varphi$-divergence into \eqref{eq:phi-arnold} and using the relation \eqref{eq:be-relation}.
\begin{theorem}
    [Functional inequality for pure Otto-Wasserstein: $\R^d$]
    Suppose the\\
    Bakry-\'Emery condition~\eqref{eq:BE-cond} holds for the target measure $\pi$.
    Then, under the $\phiP$-divergence energy for $p\in[1,2]$, the \Loj inequality~\eqref{eq:Loj-W} holds for the Otto-Wasserstein gradient flow
    with the constant $2c_\textrm{BE}$.
\end{theorem}
In addition, when the domain is an open and bounded subset of $\R^d$, we no
longer need \BE or LSI type conditions when working with sufficiently smooth
measures.

\begin{theorem}
[Functional inequality for pure Otto-Wasserstein: bounded domain]
Assume that $\Omega\subset \R^d$ is an open and bounded Lipschitz domain and
that $\pi \in \rmL^\infty(\Omega)$ is bounded from below by a positive
constant.  Then, for all $p\geq 1-\frac2d$ there exists a positive constant
$c >0$ such that
the \Loj inequality \eqref{eq:LogSob1} holds for all sufficiently smooth
measure $\mu\in \calP(\Omega)$.
    \label{thm:LogSob1}
\end{theorem}
\begin{proof}
    For the case $\pi=c_0\cdot \dd x$, the result is established in
  \citep[Sec.\,3]{mielkeConvergenceEquilibriumEnergyReaction2018} as well as the master thesis of the second author.  
The general case
follows by estimating $\pi$ from above and from below and by applying the
result to $\textstyle r=\frac{\dd\rho}{\dd\pi}$. 
\end{proof}

The important question lingering
is when and if the \Loj inequality holds for other gradient flows and other energy functionals,
\ie a theory mirroring the \BE results but going beyond the standard Otto-Wasserstein geometry.
Our starting point is 
replacing the Otto-Wasserstein geometry
of the
gradient flows
with
the Hellinger geometry.

\subsection{Hellinger gradient flow over $\Mplus$}
By a derivation similar to the Otto-Wasserstein setting, we find the energy dissipation for the Hellinger gradient flow
\begin{align}
    \calI(\mu)
    =
    -\frac{\dd}{\dd t}F(\mu(t))
    \overset{\text{(along HeGF)}}{=}
    \int \mu \biggl|\dFdmu\biggr|^2
    .
    \label{eq:energy-dissipation-fr}
\end{align}
In the settings other than Otto-Wasserstein gradient flow,
however, the \Loj inequality \refLojas
cannot be expected to hold globally for arbitrary geometry in
general.  We now show that this is precisely the case for Hellinger.  Consider
the Hellinger gradient flow with the KL entropy energy functional, \ie
$F(\mu)=\mathrm{D}_\mathrm{KL}(\mu|\pi)$.  Then, the specialized \Loj inequality
asks for the existence of some $c>0$ such that
\begin{align}
    \biggl\|\log \frac{\dd\mu}{\dd \pi}\biggr\|^2_{L^2_{\mu}}
    \geq 
    c\cdot \operatorname{\mathrm{D}_\mathrm{KL}}(\mu\|\pi)
    .
    \label{eq:local-KL-LSI-Loj}
\end{align}
\begin{figure}[htb]
    \centering
    \includegraphics[width=0.6\linewidth]{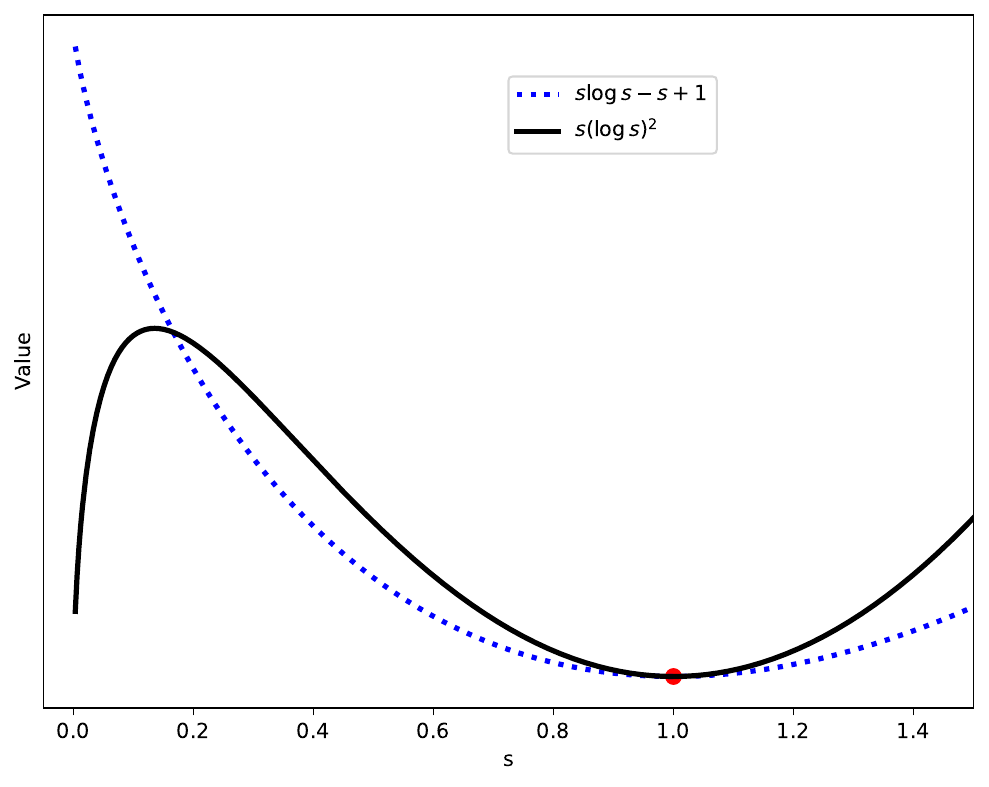}
    \caption{
    The plot illustrates the lack of global \L{}ojasiewicz inequality
    as in Lemma~\ref{lm:local-vs-global-Loj}.
    We plot the KL-entropy generator function \( \varphi(s) = s\log s -s + 1 \). 
    The blue dotted curve represents the KL-entropy generator \( \varphi(s)  \).
    The function \(  s|\log s |^2 \) is plotted in solid black. 
    The \Loj inequality condition is satisfied locally around the equilibrium \( s = 1 \) (red dot). 
    However, it can never be satisfied in a neighborhood around $s=0$.
    }
    \label{fig:lojasiewicz-kl}
\end{figure}
\begin{lemma}
    [No global \Loj condition in Hellinger flows of KL]
    There\\
    exists no $c>0$ such that \eqref{eq:local-KL-LSI-Loj} holds globally for positive measures $\mu\in\Mplus$,
    \ie
    the gradient system $(\Mplus, \mathrm{D}_\mathrm{KL}(\cdot|\pi), \He)$ does not satisfy the global \Loj condition for any positive
    constant.
    \label{lm:local-vs-global-Loj}
\end{lemma}
For a counter-example, consider $\mu_r=r\pi$ in \eqref{eq:local-KL-LSI-Loj}.
Then we have
$\mathrm{D}_\mathrm{KL}(\mu_r|\pi) \to \mu_r(\Omega)$ for $r \to 0^+$, but
$\|\log r\|_{L^2_{\mu_r}}\| = r \log^2 r \cdot \pi(\Omega)\to 0$.
See
Figure~\ref{fig:lojasiewicz-kl} and the caption for an illustration of Lemma~\ref{lm:local-vs-global-Loj}.  Despite
this lack of the global \Loj condition in general, a local condition can be
satisfied trivially around the equilibrium measure $\mu=\pi$.  However, from
this paper's perspective, we are not interested in the local version for the
reason explained next.
\begin{example}
    [Birth escaping zero in the Hellinger geometry]
    Suppose we wish to minimize the energy $F(\mu)= \mathrm{D}_\mathrm{\varphi}(\mu|\pi)$ starting from the initial measure $\mu_0$.
    It is possible that the measure $\mu_0$ does not have the full support as the target measure $\pi$; see Figure~\ref{fig:escape-zero-fr} (left), \ie
    $\operatorname{supp}( \mu_0) \subsetneq \operatorname{supp}( \pi)$.
    In addition, many variational inference methods, \eg \citep{khan2018fast,lambertVariationalInferenceWasserstein2022} use Gaussian densities to approximate the target measure.
    In such cases, the measures share the support as in Figure~\ref{fig:escape-zero-fr} (right), \ie
    $\operatorname{supp}( \mu_0) =\operatorname{supp}( \pi)$, 
    but the density ratio can be arbitrarily close to zero.
    For example, Figure~\ref{fig:escape-zero-fr} (right) depicts a Gaussian mixture distribution as the initial $\mu_0$ that has very little mass near $x=2$.
    This can be quite likely in high dimensions.
    The most difficult part of the minimization is to \emph{escape the near-zero region with enough metric slope provided by the energy}.
    For example, the reaction dynamics $\dot \mu = - \mu \dFdmu$ implies that a significant growth field is needed to escape when $\mu$ is near zero,
    \ie the \emph{birth process}.
    Our theory precisely characterizes this escape threshold via the global \Loj condition, \eg in Corollary~\ref{cor:loj-power}.
    In contrast, the local convergence behavior near the equilibrium is much easier to capture; see Figure~\ref{fig:lojasiewicz-kl}, Figure~\ref{fig:LHS-slopes-loj}.
    Therefore, we place the focus of our analysis
    on the global \Loj condition without delving into local equilibrium behavior.
    Finally, we note that our analysis is for general positive measures.
    The submanifolds of parameterized probability distributions, such as Gaussian densities (i.e. Fisher-Rao geometry), are not considered in this paper.
    \label{ex:escape-zero}
\end{example}
\begin{figure}[h!]
    \centering
    \begin{minipage}[b]{0.49\linewidth}
        \centering
        \includegraphics[width=\linewidth]{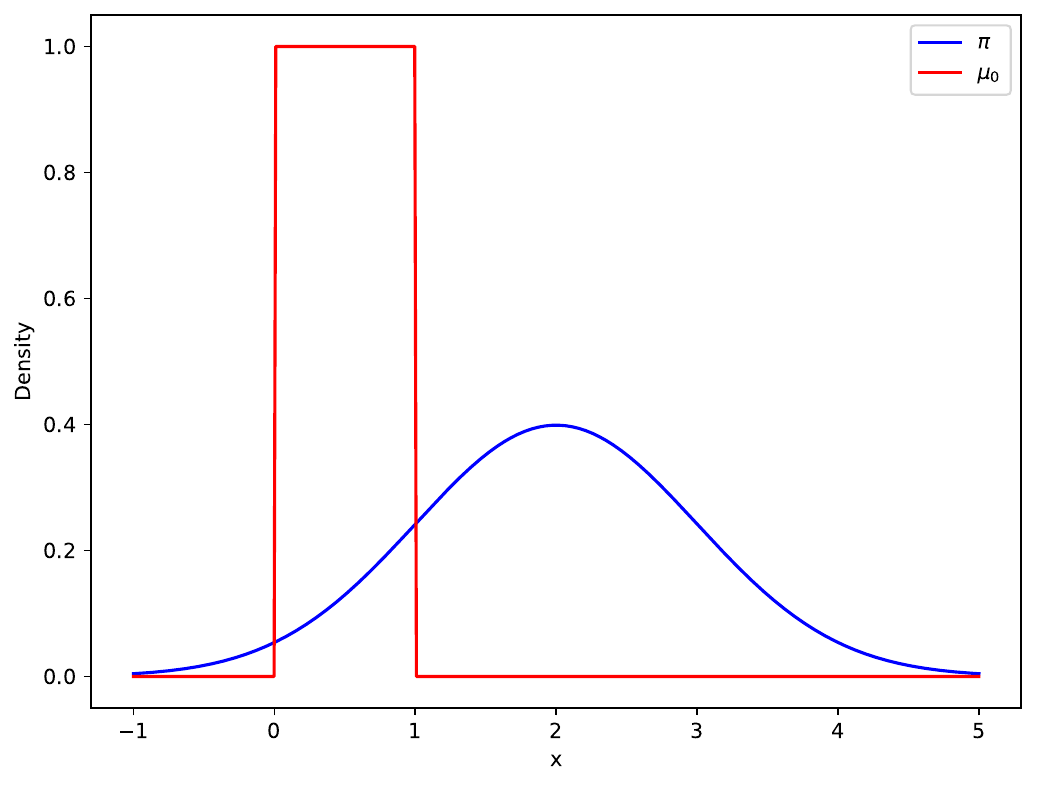}
    \end{minipage}
    \begin{minipage}[b]{0.49\linewidth}
        \centering
        \includegraphics[width=\linewidth]{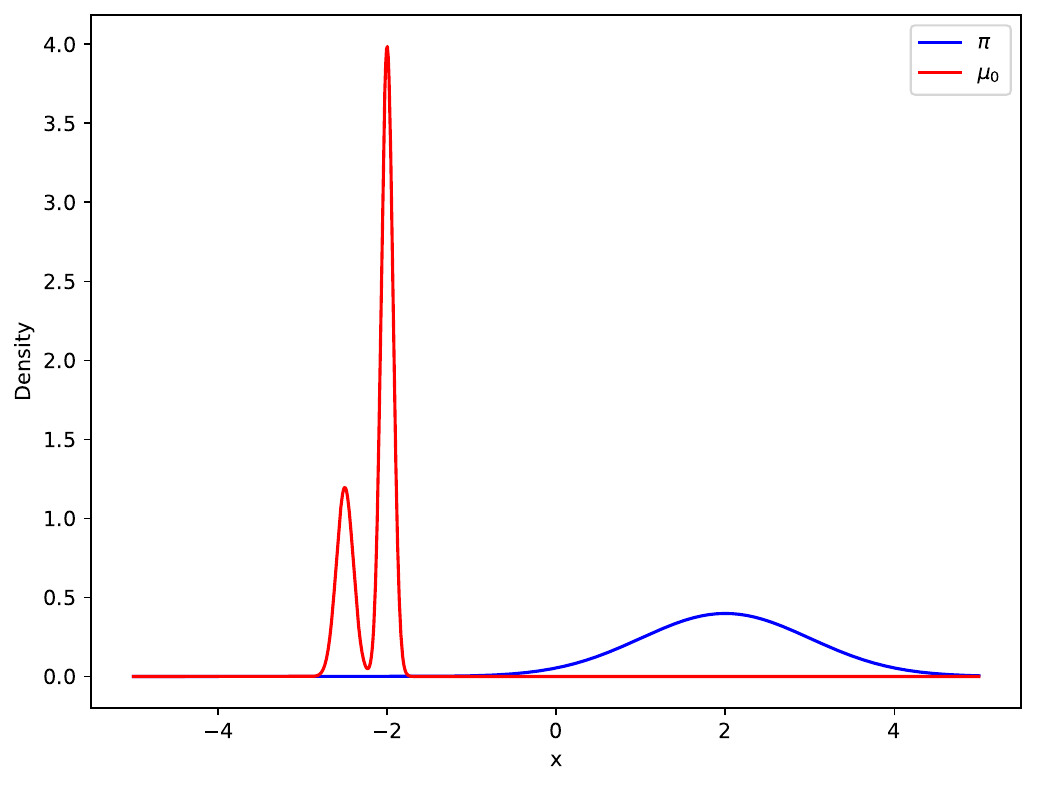}
    \end{minipage}
    \caption{Illustration of Example~\ref{ex:escape-zero}: birth escaping (near) zero
     with initial densities $\mu_0$ (red) and target densities $\pi$ (blue, Gaussian).
    }
        \label{fig:escape-zero-fr}
\end{figure}

\paragraph*{}
While the above results show that
the Hellinger flow of the KL-entropy cannot satisfy the global
\Loj, 
we now show a positive result for the case when the energy functional is the squared Hellinger distance,
$F(\mu) =\frac12\He^2(\mu, \pi) = \int\Big( \sqrt{\frac{\rmd\mu}{\rmd\pi}}
-1\Big)^2 \dd \pi$.
First, note that the first variation of the squared Hellinger distance
is 
$\displaystyle \ddmu\left( \frac12\He^2(\mu, \pi)\right)  = f'(\mu) = 2-2\sqrt{{\dd\pi}/{\dd\mu}}$.
Specializing the \Loj inequality to this setting requires
\begin{align}
4\cdot \Big\|1-\sqrt{\frac{\dd\pi}{\dd\mu}}\Big\|^2_{L^2_\mu}\geq \frac{c}2\He^2(\mu|\pi) 
.
\label{eq:LSI-FR-energy-FR-geom}
\end{align}
It can be easily checked by definition that we have the unconditional
satisfaction of the global \Loj inequality in this case:
\begin{lemma}
[Global \Loj with Hellinger energy]
The \Loj inequality \eqref{eq:LSI-FR-energy-FR-geom} holds for the Hellinger gradient system $(\Mplus, \frac12\He^2(\cdot, \pi), \He)$
globally for $c = 2 $.
\label{lm:loj-fr-energy}
\end{lemma}
Going beyond the Hellinger energy,
we are now ready to extract some general principles.
The natural question is whether relations such as \BE and \eqref{eq:LSI} exist for the Hellinger geometry. 
To answer that, we first establish the condition for global \Loj condition
for the class of $\phiP$-divergence energy~\eqref{eq:phi-p-div}.
We observe that $\mathrm{D}_\phiP(\mu|\pi)\geq 0$ with equality if and only if $\mu=\pi$
(in the sense of measures). Moreover, $\mu \mapsto \mathrm{D}_\phiP(\mu|\pi)$ is convex,
and the Fr\'echet subdifferential is given by
\[
\pl \mathrm{D}_\phiP(\mu|\pi) 
= D \mathrm{D}_\phiP(\mu|\pi) 
= \DphiP\left(\ddfrac{\mu}{\pi}\right)= \frac1{p{-}1} \left(
\left(\ddfrac{\mu}{\pi}\right)^{p-1} - 1\right)
,\quad
D \rmD_{\varphi_1}(\mu|\pi)=\log\left(\ddfrac{\mu}{\pi}\right).
\]

\begin{proposition}
    [Global \Loj for Hellinger gradient flow of relative entropy]
\label{prop:loj-phi}
Given
the Hellinger gradient system with $\varphi$-divergence energy, \ie
$\left(
    \Mplus,
    \mathrm{D_\varphi}(\cdot | \pi),
    \He
\right)$.
If $\varphi:(0,\infty)\to [0,\infty)$ is a convex entropy generator function satisfying
\begin{equation}
    \varphi(1)=\varphi'(1)=0, \varphi''(1)>0 \quad \text{and} \quad \exists\, c_*>0\text{ such that } \forall\, s>0: \ s \bigl(\varphi'(s)\bigr)^2 \geq c_* \varphi(s),
        \label{eq:FR.Loja.Cond}
\end{equation}
then the \Loj inequality holds globally, \ie
    \begin{align}
        \biggl\|\varphi^\prime \left(\frac{\dd\mu}{\dd \pi}\right)\biggr\|^2_{L^2_{\mu}}
        \geq 
         c_*  \operatorname{\mathrm{D}_\varphi}(\mu|\pi) .
        \label{eq:Loj-FR}
        \tag{\L{}-He}
    \end{align}
\end{proposition}
\begin{proof}
    [Proof of Proposition~\ref{prop:loj-phi}]
    As previously calculated,
    the first variation of the $\varphi$-divergence is given by
  $\displaystyle\ddmu \mathrm{D_\varphi}(\mu|\pi) = \varphi^\prime \left(\frac{\dd\mu}{\dd \pi}\right) $.
Thus, using the Hellinger metric, we obtain the dissipation relation $\frac{\dd}{\dd t} \mathrm{D_\varphi}(\mu|\pi) = -\calI(\mu) $ with
\[
  \calI(\mu) = \left\| \varphi^\prime \bigl(\tfrac{\dd\mu}{\dd \pi}\bigr)\right\|_{L^2_\mu}^2
  = \int_\Omega \left(\varphi^\prime \bigl(\tfrac{\dd\mu}{\dd \pi}\bigr)\right)^2 \dd \mu 
  = \int_\Omega \left(\varphi^\prime \bigl(\tfrac{\dd\mu}{\dd \pi}\bigr)\right)^2 
     \tfrac{\dd\mu}{\dd \pi} \dd \pi . 
\]
Now exploiting the assumption \eqref{eq:FR.Loja.Cond} for estimating the integrand, we immediately obtain \eqref{eq:Loj-FR}. 
\end{proof}
Because of the simple point-wise estimate in the above proof, it is also clear that condition \eqref{eq:FR.Loja.Cond} is \emph{necessary and sufficient} for the \Loj estimate \eqref{eq:Loj-FR}. 
\begin{corollary}
    [Hellinger gradient flows: necessary sufficient condition]
    The\\
    \Loj inequality \eqref{eq:Loj-FR} for the Hellinger gradient system with the power-like entropy $\phiP$~\eqref{eq:power-ent} energy, $(\Mplus, \mathrm D_{\phiP}, \He)$, holds globally if and only if $p\leq\frac12$.
    Furthermore, the constant is $c_*=1/(1{-}p)$ in that case.

    Therefore,
    the Hellinger gradient flow under the
    $\phiP$-divergence energy functional
    decays exponentially globally,
    \ie 
    \begin{align*}
        \rmD_{\phiP}(\mu(t)|\pi) \leq \rme^{-\frac{t}{(1-p)}} 
        \cdot \rmD_{\phiP}(\mu(0)|\pi)
    \end{align*}
    if and only if $p\leq\frac12$.
    \label{cor:loj-power}
\end{corollary}
This decay result is also referred to as 
global exponential convergence in energy.
In short, for the $\varphi$-divergence energy functional,
\begin{align*}
    \eqref{eq:FR.Loja.Cond}\iff 
    \eqref{eq:Loj-FR}
    \implies
    \textrm{exp. decay}
\end{align*}
\begin{figure}[ht]
    \centering
    \includegraphics[width=0.7\linewidth]{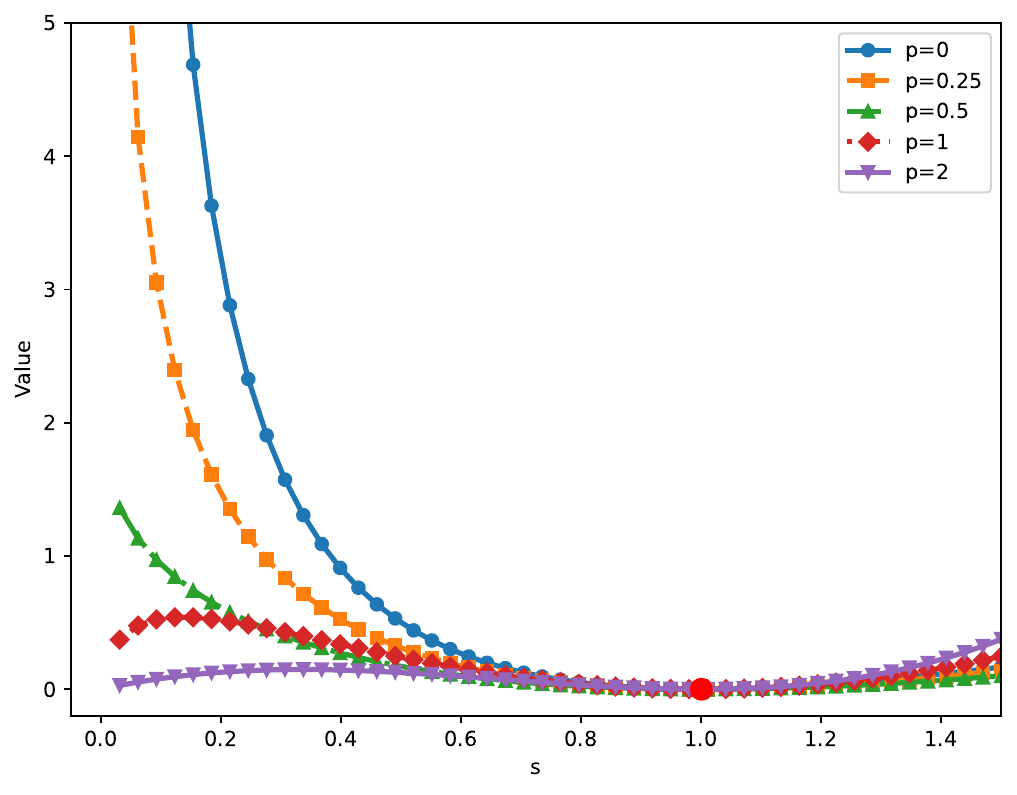}
    \caption{
    The plot illustrates the left-hand side $s(\phiP'(s))^2$ of the
    \L{}ojasiewicz inequality~\eqref{eq:FR.Loja.Cond} for the Hellinger geometry for $s\in [0,1.2]$ and different $p$: purple $p=2$ ($\chi^2$),
    red $p=1$ (KL),
    green $p=0.5$ (Hellinger),
    orange $p=0.25$,
    blue $p=0$ (forward KL),
    The red dot represents the equilibrium at $s=1$, where $\phi'(s)=0$. 
    This plot provides insights into the slopes of the power-like 
    entropies in the Hellinger gradient flow. Indeed,   
    Proposition \ref{prop:loj-phi} discusses the relation of the corresponding curves $\phiP(s)$ in Figure \ref{fig:power-ents} 
    with those here.  
    We observe the threshold \(p=0.5\) (Hellinger; green) where the behavior near $s=0$ jumps.
    See the main text, especially Remark~\ref{rem:power-threshold}, for analysis.}
    \label{fig:LHS-slopes-loj}
\end{figure}
In particular, Corollary~\ref{cor:loj-power} shows that the energy functionals, for 
which the globally \Loj estimate holds, include the squared Hellinger 
($p=\frac12$), the forward KL ($p=0$), 
the reverse $\chi^2$ ($p=-1$),
and the 
fractional-power entropies between those. On the negative side, it states that 
the \Loj estimate does not hold globally for many commonly used entropy 
functionals such as the KL ($p=1$) and $\chi^2$ ($p=2$).

\begin{remark}
    [Metric slope and entropy power threshold $p=\frac12$]
    The relevance of the threshold $p= 1/2$ can be seen from two perspectives.
    First, we 
    observe that $\mu=0$ is a steady-state solution for the gradient systems $ \big( \Mplus,\rmD_{\phiP}(\cdot |\pi),\He \big)$ for $p>1/2$. However, if $\mu(t)=0$ 
    is a solution, then it cannot converge exponentially to the equilibrium 
    measure $\pi$. The point is that the
    \emph{Hellinger metric slope},
    defined as
    \[
    |\partial \rmD_{\phiP}|_\He(0) := \limsup_{\mu \to 0} 
    \frac{\big(\rmD_{\phiP}(0)- \rmD_{\phiP}(\mu)\big)_+}
    {\He(0,\mu)}
    ,
    \]
    can be calculated explicitly
    as the following:
    \begin{lemma}
        The Hellinger metric slope
        of the $\phiP$-divergence energy functional
        at $\mu=0$ is given by
        \[
        |\partial \rmD_{\phiP}|_\He(0) =\begin{cases}
            0 & \text{for }p>1/2,\\ 
            1& \text{for } p=1/2, \\ \infty &\text{for } p<1/2.
        \end{cases}
        \]
        \label{lm:Hellinger-metric-slope}
    \end{lemma}
    In the case $p>0$ where $\rmD_{\phiP}(0) <\infty$ the curve $t \mapsto \mu(t)=0$ can still be considered a solution of the gradient-flow equation, however, the exponential decay only applies to the curves of maximal slopes (see, \eg \citep{ambrosio2008gradient}) satisfying
    the dissipation balance
    \[
    \frac{\dd }{\dd t } \rmD_{\phiP} (\mu(t)|\pi)=- \frac12|\mu'|_\He(t)^2 
    - \frac12 |\partial  \rmD_{\phiP}|_\He(\mu(t))^2 .
    \]
    We refer to \cite[Section~2]{LasMie23EVIH} for a more detailed discussion.
    
    A second way to see the importance of the threshold $p\leq\frac12$ involves the results in 
    \citep{otto2000generalization} showing that geodesic $\Lambda$-convexity of a functional implies the \Loj inequality with $c_\text{\L}= 2\Lambda$.
    This is analogous to
    the finite-dimensional case in \citep{karimiLinearConvergenceGradient2020}.
    For the condition of geodesic $\Lambda$-convexity for functionals $\rmD_\varphi(\mu|\pi) = \int_\Omega \varphi(\frac{\dd\mu}{\dd \pi}) \dd \pi$ in the Hellinger geometry, it can be shown that
    $\Lambda := \inf_{w\geq 0} \Big\{  w \,\varphi''(w) + \frac12 \, \varphi'(w)\Big\}$. 
    This gives the same result when considering the $p$-power family $\phiP$. But for general $\varphi$, we may have $2\Lambda \lneqq c_\text{\L}$. 
    The geodesic $\Lambda$-convexity in the HK geometry
has been established in
\citep[Theorem~7.2]{LiMiSa23FPGG}.
However,
one can only obtain a \Loj type inequality with constant zero by directly applying the results in \citep{LiMiSa23FPGG},
which is not sufficient for exponential convergence.
\label{rem:power-threshold}
\end{remark}

\subsubsection{Explicit solution of the Hellinger gradient flow equation}
To further characterize the phenomena
regarding the HK gradient flow
mathematically,
we now delve deeper into
the gradient-flow equation for the Hellinger
gradient systems $(\Mplus,F(\mu):=\rmD_{\varphi_p}(\cdot |\pi),\He)$, which is the reaction equation
\begin{align}
\label{eq:GFE.Dp.He}
 \dot\mu = -\beta\cdot  \mu  \, \frac{\delta \mathrm{D}_\phiP(\mu|\pi)}{\delta \mu }[\mu]
 = -\beta\cdot  \mu \ \DphiP(\frac{\dd \mu}{\dd\pi})
 .
\end{align}
A delicate situation arises when considering the gradient flow equation~\eqref{eq:GFE.Dp.He}:
in general, not all solutions of \eqref{eq:GFE.Dp.He} will
converge to the desired equilibrium $\pi$. The reason is the degeneracy of the
Hellinger Onsager operator 
$\bbK_{\He}(\mu)\xi=\mu\cdot \xi$ at $\mu=0$. 

\begin{example}
    [Hellinger gradient flow of KL]
Taking the driving energy to be the KL divergence ($p=1$) in
\eqref{eq:GFE.Dp.He}, we obtain the gradient flow equation
\begin{align}
  \dot\mu= - \beta \mu \,\log\left(\frac{\dd \mu}{\dd\pi}\right).
  \label{eq:GFE.KL.He}
  \tag{KL-He}
\end{align}
This ODE can be explicitly solved with elementary arguments, yielding the
following result.
\end{example}

\begin{proposition}
  The ODE~\eqref{eq:GFE.KL.He} admits the unique solution, for all
  $x\in \Omega$ and $t\geq 0$,
  \begin{align}
    \mu(t,x) = \pi(x) \Big( \tfrac{\ds \rmd\mu(0,\cdot)}{\ds \rmd\pi} (x) 
    \Big)^{\ds \ee^{-\beta t}}
  \end{align}
  Furthermore, we have $\mu(t,x) \to \pi(x)$ as $t\to \infty$ if and only if
  $\mu(0,x) \gneqq 0$.
\end{proposition}
In other words, for some location $x'$ with zero initial density $\mu(0,x')=0$,
the solution gets stuck and no new mass is born.  This precisely corresponds
the illustration in Figure~\ref{fig:lojasiewicz-kl}.  In addition to the KL
divergence functional ($p=1$), similar problems occur for the Hellinger
gradient flow of the $\rmD_{\varphi_p}$-divergence with $p\in {(0,1)}$, because
solutions starting with $\mu(0,x)=0$ may satisfy $\mu(t,x)=0$ for
$t\in [0,\tau(x)]$ and $\mu(t,x)>0$ for $t> \tau(x)$, where $\tau(x)$ can be
chosen arbitrarily.
However,
for the interesting case of $\varphi_p$-divergence with $p\leq 1/2$, the notion of \emph{curves
of maximal slope} selects the unique solution with $\mu(t,x)>0$ for $t>0$.

\subsubsection{Exponential decaying Lyapunov functions for Hellinger gradient flows}
Clearly, 
the driving energy
$\rmD_{\varphi_p}(\,\cdot\, | \pi)$
itself 
decays along solutions because it is the driving energy of the gradient
system. 
Furthermore, an examination of the simple structure of the gradient flow equation~\eqref{eq:GFE.Dp.He} implies that $\dot \rho \geq
0$ for $\frac{\mathrm{d}\rho}{\mathrm{d}\pi} \in [0,1]$ and $\dot \rho\leq 0$ for $\frac{\mathrm{d}\rho}{\mathrm{d}\pi}\geq 1$. Hence,
the divergence functional $\rmD_{\varphi_q}(\,\cdot\, | \pi)$ with any $q$ is non-increasing along solutions. 
We have 
\[
\frac\rmd{\rmd t} \rmD_{\varphi_q}(\mu(t)|\pi) = -\beta \int_\Omega 
\underbrace{\mu \DphiP\left(\ddfrac{\mu}{\pi}\right) \varphi'_q \left( 
\ddfrac{\mu}{\pi}\right)}_{\geq 0} \dd x \ \leq \ 0  ,
\]
\ie $\rmD_{\varphi_q}(\,\cdot\, | \pi)$ is a Lyapunov functional for the
Hellinger gradient flow of the $\varphi_p$-divergence.

We now show that the divergence $\rmD_{\varphi_q}( \cdot |\pi)$, for some $q\neq p$, decays exponentially along
the gradient flow solutions.
Set an auxiliary constant depending on $p$ and $q$ as
\[
m_{p,q} := \inf\Bigset{\frac{r \DphiP(r)\varphi'_q(r)}{ \varphi_q(r)}}{r>0 \text{
    and } r\neq 1} \geq 0.
\] 

\begin{proposition}
[Lyapunov functionals for $(\Mplus,\rmD_{\varphi_p},\He)$]
\label{prop:lyapunov-q}
For $p, q \in \R$, we have
\[
m_{p,q} \gneqq 0 \ \Longleftrightarrow \ p\leq \max\big\{ 0, \min\{ 1,1{-}q\}
\big\}.
\]
Assume that the initial condition $\mu(0)$ satisfies
$\rmD_{\varphi_q}(\mu(0)|\pi)<\infty$ and $m_{p,q}>0$.
Then $\rmD_{\varphi_q}$ decays
exponentially along the solutions of the gradient flow for
$(\Mplus,D_{\phiP}(\cdot|\pi ),\He)$, namely
  \[
      \rmD_{\varphi_q}(\mu(t)|\pi) \leq \ee^{-\beta m_{p,q} t} \rmD_{\varphi_q}(\mu(0)|\pi) \quad \text{for } t
    \geq 0.
  \]
\end{proposition}
That is, the $\varphi_q$-divergence is an exponentially decaying Lyapunov functional for the Hellinger gradient flow of the $\varphi_p$-divergence.
\begin{proof}
The technical characterization of the region with $m_{p,q}>0$ is given in
Lemma~\ref{lm:m-pq-relations}.

For the decay estimate \EEE we simply observe that 
  \begin{align*}
  -\frac{\rmd}{\rmd t} \rmD_{\varphi_q}(\mu(t)|\pi) &=  \beta \int_\Omega 
  \DphiP\big(\tfrac\mu\pi\big) \,\varphi'_q\big( \tfrac\mu\pi\big) \dd\mu =
   \beta \int_\Omega 
  \DphiP\big(\tfrac\mu\pi\big) \,\varphi'_q\big( \tfrac\mu\pi\big) \,
  \tfrac\mu\pi \, \dd \pi 
    \\ & 
  \geq \int_\Omega m_{p,q} \varphi_q\big(\tfrac\mu\pi\big) \dd \pi = \rmD_{\varphi_q}(\mu(t)) .
  \end{align*}
Now, the desired result follows by Gr\"onwall's estimate. 
  \end{proof}  

\begin{figure}[htb]
    \centering
    \includegraphics[width=0.7\linewidth]{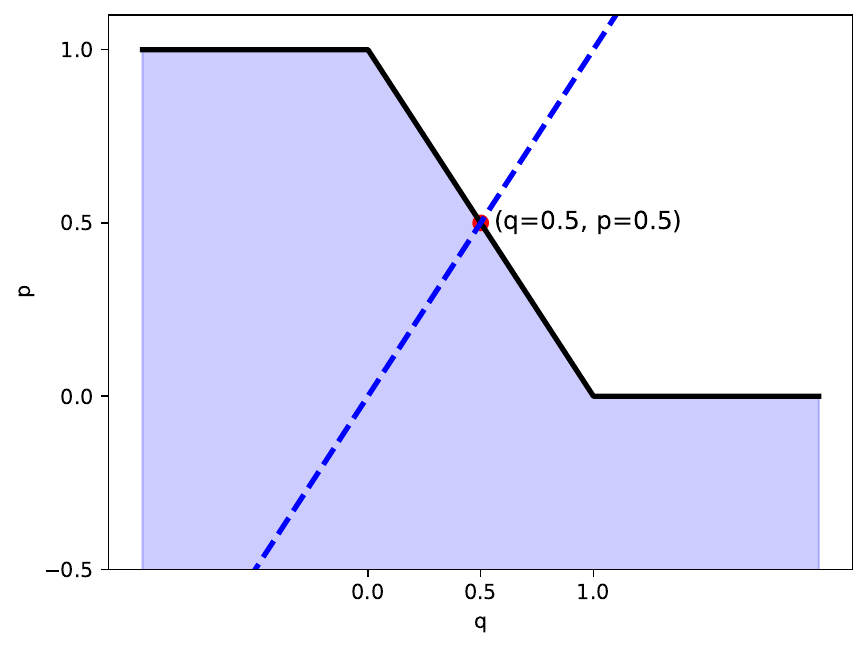}
    \caption{
    The plot illustrates the $p$ and $q$ values that satisfy the condition $p\leq \max\big\{ 0, \min\{ 1,1{-}q\} \big\}$.
    The shaded area represents the region where $p\leq \max\big\{ 0, \min\{ 1,1{-}q\} \big\} \iff m_{p,q}>0$,
    \ie the $\varphi_q$-divergence converges exponentially
    for the Hellinger gradient flow of the $\varphi_p$-divergence.
    See Proposition~\ref{prop:lyapunov-q} and Lemma~\ref{lm:m-pq-relations}
    for the details.
    Furthermore, we observe that the shaded area contains the part of the line $p=q$ for $q\leq \frac12$.
    This shows that Proposition~\ref{prop:lyapunov-q} generalizes the result in Corollary~\ref{cor:loj-power}
    to exponentially decaying Lyapunov functionals.
    In this case, when $q>\frac12$, the intersection is empty and hence our result no longer implies $\rmD_{\varphi_p}$ itself decays exponentially.
    }
    \label{fig:mpq-fcn}
\end{figure}

The proof above and
Lemma~\ref{lm:m-pq-relations}
can further explain why it is easier to have decay estimates for the
divergences $\rmD_{\varphi_q}(\mu(t)|\pi)$ with $q{\leq }0$.
Since, if $q >0$,
the condition
$\rmD_{\varphi_q}(\mu(0)|\pi)<\infty$ will impose strict positivity of the density $\mu(0,x)>0$ a.e.\
with respect to $\pi$. 
For example, the forward KL divergence ($q=0$) does not impose this condition.
\begin{example}
    [Lypunov functional for Hellinger gradient flows of KL]
    We again\\
    consider the driving energy of a Hellinger gradient flow
    to be the KL divergence ($p=1$).
    Due to Proposition~\ref{prop:lyapunov-q},
    the forward KL ($q=0$) divergence, $\rmD_{\varphi_0}(\cdot |\pi) = \rmD_\textrm{KL}(\pi | \cdot )$,
    and the reverse $\chi^2$ ($q=-1$) divergence, $\rmD_{\varphi_{-1}}(\cdot |\pi)=\rmD_{\chi^2}(\pi | \cdot )$,
    are both exponentially decaying
    Lyapunov candidates for this system, \ie they decay exponentially along the Hellinger gradient flow of the KL divergence, given the finite initialization condition in Proposition~\ref{prop:lyapunov-q}.
    See also Figure~\ref{fig:mpq-fcn} for the relation between $p$ and $q$.
\end{example}

\EEE

\section{Spherical HK gradient flows of probability measures~$\calP$}
\label{sec:SHK}

Our main goal of this section is to advance the state-of-the-art analysis for the spherical
Hellinger-Kantorovich (a.k.a., spherical Wasserstein-Fisher-Rao) space and
gradient flows of probability measures $\calP$.  Its properties differ from the
HK (a.k.a. WFR) geometry over positive measures $\Mplus$.  Remarkably, we
are able to establish a global Polyak-\Loj inequality for the
$\phiP$-divergence energy when $p \in (-\infty,\frac12]\cup [1,\infty)$, which
showcases the advantages of the SHK geometry over the pure Otto-Wasserstein and
Hellinger geometries.  
This is due to the flexibility of SHK by combining the strengths of
the Otto-Wasserstein and the spherical Hellinger geometries over probability
measures.
We note that a detailed and insightful analysis of
the gradient flow for $(\Mplus(\Omega),\rmD_\varphi(\cdot|\pi), \SHe)$ is also
contained 
in \citep{carrilloFisherRaoGradientFlow2024}. In particular, sufficient conditions for geodesic
convexity are presented, and a necessary and sufficient condition for the
\Loj inequality (called ``gradient dominance'' therein) are derived. Our
results are different in the sense that we look for general Lyapunov functions
such as $\rmD_{\varphi_q}(\cdot| \pi)$, where $q \neq p$ is allowed, whereas their focus is on the decay of the sum of certain two entropy functionals.

\subsection{Pure spherical Hellinger gradient flow of probability measures}
\label{sec:pure-she}
 The gradient-flow equation for
the gradient system $(\calP(\Omega), \mathrm{D}_\phiP(\,\cdot\,|\pi), \SHe)$ takes the form
\begin{equation}
  \label{eq:GFE.Dp.Bh}
  \dot\rho = - \beta \rho\left( \DphiP\left(\frac{\dd\rho}{\dd\pi}\right) - 
\int_{\Omega} \DphiP\left(\frac{\dd\rho}{\dd\pi}\right) \dd\rho \right),
\end{equation}
where we have used the letter $\rho\in\mathcal P$ for probability measure
instead of the positive measure $\mu\in\Mplus$.  The following result
establishes a Polyak-\Loj inequality for the pure SHe gradient flow, which
reads
\begin{equation}
    \label{eq:PolLoj.Bh}
    \int_{\Omega} \DphiP\left(\frac{\dd\rho}{\dd\pi}\right) \,\bbK_{\SHe}(\rho) 
\DphiP\left(\frac{\dd\rho}{\dd\pi}\right) \dd x \geq  \beta \,M_p\,
D_p(\rho|\pi) 
\text{ for some }
M_p >0
\end{equation}
for the case $p\in [0,1/2]$ that leads to exponential decay of solutions for
$t>0$.  Recall the definition of the spherical Hellinger Onsager operator
$\bbK_{\SHe}$ in \eqref{eq:SHK-bbK} and use $\int_\Omega \dd \rho=1$, we
establish the following result.

\begin{theorem}[\L-$\SHe$ estimate]
\label{thm:PolyakLoj.Bh}
Assume $p\in (-\infty,1/2]$ and that $\rho \in \calP(\Omega)$ satisfies
$\rho(x)>0$ a.e.\ with respect to $\pi$. Then, the following functional inequality holds,
\begin{align}
  \int_{\Omega} \ddfrac{\rho}{\pi}\,\left(
    \DphiP\left(\ddfrac{\rho}{\pi}\right)\right)^2 \dd \pi -
  \left(\int_{\Omega} \ddfrac{\rho}{\pi} \DphiP\left(\ddfrac{\rho}{\pi}\right)
    \dd  \pi \right)^2 
  \geq M_p \int_{\Omega} \phiP\left(\ddfrac{\rho}{\pi}\right) \dd \pi
  \label{eq:Loj-SH}
  \tag{\text{\L}-SHe}
\end{align}
with $
M_p =  \begin{cases} \frac1{1{-}p}& \text{for }  p\leq \frac13, \\[0.2em]
  \frac{p(7{-}12p)}{1{-} p} &\text{for } p \in [\frac13,\frac12]. \end{cases}$  
For $p>1/2$ the best possible constant is $M_p=0$.
\end{theorem}
For $p>1/2$, we do not have convergence because there are multiple steady
states, namely $\rho_\text{steady}(x) = \frac1{\pi(A)} \bm1_A(x)$ for arbitrary
sets $A\subset \Omega$ with $\pi(A)>0$. 
\begin{proof} 
[Proof of Theorem~\ref{thm:PolyakLoj.Bh}]
We use the abbreviation $r= \ddfrac{\rho}{\pi}$ such that $ r \geq 0 $ a.e.
We treat the case $p\in {(0,1/2]}$ first. We now use the definition of $\phiP$ \eqref{eq:power-ent} and the abbreviation $I_\alpha(r)=
\int_{\Omega} r^\alpha \dd \pi$.  Since $\rho,\pi \in \calP(\Omega)$ we have 
\begin{align}
        I_0(\rho)=  I_1(\rho)=1  \quad \text{and} \quad \mathrm{D}_\phiP(\rho|\pi) = \frac{1 - I_p(r)}{p(1{-}p)} \geq 0,
        \label{eq:helper-div-posi}
\end{align}
which implies $I_p(r)\leq 1$.

For the left-hand side of \eqref{eq:Loj-SH},
we obtain, after some major cancellations, the simple relation 
\[
\text{LHS} = \frac1{(1{-}p)^2} \big( I_{2p-1}(r) - I_p(r)^2\big) .
\]
Moreover, H\"older's inequality gives 
\[
I_{\alpha+ \beta}(r) \leq I_{\alpha/\theta }(r)^{\theta} I_{\beta/(1-\theta)
}(r)^{1-\theta} \quad\text{for } \alpha,\beta\in \R \text{ and } \theta \in
{(0,1)}.
\] 
We use $I_0(\rho)=1$ and choose $\alpha,\ \beta$, and $\theta$ such that
$\alpha+\beta=0$, $\alpha/\theta=p$, and $\beta/(1{-}\theta)=2p-1$. This gives
$\theta=(1{-}2p)/(1{-}p) \in [0,1]$ and $\alpha =-\beta=p(1{-}2p)/(1{-}p)$, and
we find  $I_{2p-1}^{p/(1-p)} I_p^{ (1-2p)/(1-p)} \geq  I_0(r) =1$ and conclude 
\[
\text{LHS}  \geq \frac{I_p(r)^{ 2-1/p } - I_p(r)^2}{(1{-}p)^2} \geq
 \frac{A_p} {(1{-}p)^2}  \big( 1-I_p(r)\big) \quad \text{with }
A_p=\begin{cases} 1/p& \text{for } p\in {(0,\frac13]}, \\ 7{-}12p& \text{for
  }p\in [\frac13, \frac12]. 
\end{cases} 
\]
The last estimate follows from the fact that
$y \mapsto g_p(y):= (y^{2-1/p} {-}y^2)/(1{-}y)$ satisfies $g_p(y) \to 1/p$
for $y\nearrow 1$. Moreover, for $p\in {(0,\frac13]}$ we have $g'_p(y) \leq 0$
  which gives $g_p(y)\geq 1/p$. For $p\in [\frac13,\frac12]$ the result can be similarly 
  checked or numerically verified.

The case $p=0$ is easier, because $\int_{\Omega} r\varphi'_0(r) \dd \xi=\int_{\Omega}(r{-}1)\dd
\xi = \int_{\Omega}\dd \rho- \int_{\Omega} \dd \pi= 0$. 
Hence, we have 
\[
\text{LHS}= \int_{\Omega}r\big( 1- \frac1r \big)^2 \dd \pi \quad \text{and} \quad 
\mathrm{D}_{\varphi_0}(\rho|\pi) = \int_{\Omega}\! {-}\log r\,\dd \pi. 
\]
We further process the left-hand side,
\begin{align*}
    \int_{\Omega}r\big( 1- \frac1r \big)^2 \dd \pi
    = \int_{\Omega} \big( r -2 + \frac1r \big) \dd \pi
    = \int_{\Omega} \left(  \frac1r  -1 \right) \dd \pi
    ,
\end{align*}
where the last equality follows from $\int_{\Omega} r \dd \pi =1$.
Then,
the desired estimate with $M_0=1$ follows from the elementary inequality
$\frac1r - 1 \geq - \log r$ for $r>0$,
\ie
\begin{align*}
    \text{LHS} = \int_{\Omega} \big(  \frac1r  -1 \big) \dd \pi \geq \int_{\Omega} \big( - \log r \big) \dd \pi = \mathrm{D}_{\varphi_0}(\rho|\pi).
\end{align*}

In the case of $p< 0$, we use an argument similar to the case of $0<p<1/2$, but taking into account $I_p(\rho) \geq 1$.
Using H\"older's inequality, we have
$I_p(\rho) \leq I_0(\rho)^{1-\theta} I_{p/\theta}(\rho)^\theta =
I_{p/\theta}(\rho)^\theta$.
Choosing $ \theta =-p/(1{-}2p)\in [0,\frac12)$
gives $I_{2p-1}(\rho) \geq I_p(\rho)^{2+1/|p|}$. 
By the convexity 
of the function $y^{2+1/|p|}-y^2$ and a Taylor expansion around $y=1$,
we have
$y^{2+1/|p|}-y^2\geq \frac{1}{|p|}(y{-}1)$ for $y\geq 1$. With $y=
I_p(\rho) \geq 1$, we find  
\[
\text{LHS} = \frac{I_{2p-1}(\rho){-}I_p(\rho)^2}{(p-1)^2} \geq
\frac{I_p(\rho)^{2+1/|p|}{-} I_p(\rho)^2}{(p-1)^2} \geq 
\frac{I_p(\rho){-}1}{|p|(p{-}1)^2} = \frac{1}{1{-}p} \,D_p(\rho| \pi), 
\]
which is the desired result.\medskip

For $p>1/2$ we consider
the measure $\rho_\eps$ such that
$\rho_\eps(x)=\eps \cdot \pi (x), \ \eps>0$ on $A_\eps$ and $\rho_\eps(x)=2 \cdot \pi (x)$
on $X\setminus A_\eps$. 
Since
$I_1(\rho_\eps)=1$ must be satisfied,
we obtain 
$\pi(A_\eps) =  1/(2{-}\eps),\ 1-\pi(A_\eps) = (1{-}\eps)/(2{-}\eps)
$.
Moreover, for $q>0$ we have $I_q(\rho_\eps) \to
2^{q-1}$ for $\eps \to 0$. Thus, for $\eps\to 0$, we obtain
\[
  \mathrm{D}_\phiP(\rho_\eps|\pi)\to  (2^p{-}1)/(p^2-p)>0, \quad I_p(\rho_\eps) \to 2^{p-1},
  \quad I_{2p-1}(\rho_\eps) \to 2^{2p-2},
\]
where the last relation uses the assumption $p>1/2$. Thus,
$\text{LHS}(\rho_\eps)\to 0$ for $\eps\to 0$, and the ratio $
\text{LHS}(\rho_\eps)/ \mathrm{D}_\phiP(\rho_\eps|\pi) \to 0$,
\ie this ratio cannot be lower bounded by a positive constant $M_p >0$.
Hence, the statement is proved.
\end{proof}
\begin{remark}
    [Hellinger flow of forward KL is mass-preserving]
    From the proof of Theorem~\ref{thm:PolyakLoj.Bh},
    we observe that
    the \Loj inequality for the spherical Hellinger gradient flow of the forward KL energy ($\varphi_p$ with $p=0$) 
    is contained
    in the case for the (non-spherical) Hellinger \Loj~\eqref{eq:Loj-FR}.
    However, 
    it must be noted that
    those two gradient flows are not the same:
    in the case of (non-spherical) Hellinger, the mass is only preserved when
    starting in the probability subspace $\calP(\Omega)$, but not otherwise. In fact, the total mass
    can be explicitly calculated
    with elementary arguments as
    $ Z(t) = 1 + e^{-\beta t} (Z(0)-1)$.
    In contrast, the SHK flows can be extended to the outside of $\calP(\Omega)$ to
    a mass-preserving flow. This is often done for the Otto-Wasserstein flow on positive measures; see Section~\ref{sec:HKGF}.
\end{remark}

\begin{corollary}
[Exponential Decay of $\mathrm{D}_\phiP$-divergence along SHe gradient flow]
\label{co:ExpDecay}   \quad \\
Assume $ p \in (-\infty,\frac12] $ and consider an initial datum
$\rho(0) \in \calP(\Omega)$ with $\mathrm{D}_\phiP(\rho(0)|\pi)<\infty$. Then,
the solution $\rho$ of \eqref{eq:GFE.Dp.Bh} with $\rho(t,x)>0$ for all $t>0$
a.e.\ with respect to $\pi$ satisfies an exponential decay estimate with
constant $M_p>0$ from Theorem \ref{thm:PolyakLoj.Bh}, namely
\[
\mathrm{D}_\phiP(\rho(t)|\pi) \leq \ee^{-\beta M_p t} \mathrm{D}_\phiP(\rho(0)|\pi) \text{ for all }t > 0.
\]
\end{corollary} 
\begin{proof}
This follows directly from \eqref{eq:PolLoj.Bh} and a Gr\"onwall estimate. 
\end{proof}

\subsection{Spherical Hellinger-Kantorovich space and gradient flows}
Finally, we apply our results for the SHe flows above to obtain the 
\emph{global functional inequality and hence
exponential convergence of the spherical Hellinger-Kantorovich gradient flow over probability measures $\cal P$}.

The specialized \Loj inequality for the SHK gradient flow reads
\begin{align}
    \label{eq:L-SHK}
    \int 
    \rho 
    \left(
         \alpha \EEE \bigg|\nabla \dFdrho\bigg|^2
        \!\!+  \beta \EEE
        \bigg|\dFdrho\bigg|^2
    \right)
    -  \beta \EEE \left(
        \int \rho\cdot  \dFdrho
    \right)^2
    \geq 
    c_* \left(
        F(\rho) - \inf_{\nu\in \Mplus} F(\nu)
    \right)
    .
\end{align}
We establish the following result.
\begin{theorem}
[Functional inequality for spherical Hellinger-Kantorovich]
\label{thm:SHK-Loj}
The\\
SHK \Loj inequality~\eqref{eq:L-SHK} holds globally with a positive constant
for the spherical Hellinger-Kantorovich (a.k.a., Wasserstein-Fisher-Rao)
gradient flow over probability measures for the $\phiP$-divergence energy for
$p\in (-\infty,\frac12]$ with $c_* = c_{\SHe} = \beta M_p>0$.

Furthermore, if the Otto-Wasserstein-\Loj inequality~\eqref{eq:Loj-W} with reference
measure $\pi$ holds with $c_{\text{\L-W}}>0$ for all
probability measures $\calP(\Omega)$, then the SHK \Loj inequality holds with 
$c_*=\alpha \;\! c_{\text{\L-W}}>0$.

Consequently, the SHK gradient flow converges globally with  exponential decay
rate  $c_*=\max\{ \alpha \;\! c_{\text{\L-W}}\, ,\: \beta \!\; M_p \}$.
\end{theorem}
Note that, for bounded Lipschitz domains
$\Omega$ and $\pi \in L^\infty(\Omega)$ with $\inf_\Omega \pi(x) >0$, 
the Otto-Wasserstein-\Loj inequality~\eqref{eq:Loj-W} indeed holds with $c_{\text{\L-W}}>0$
for $p >1{-}\frac1d$, see \citep[Sec.\,3.1]{mielkeConvergenceEquilibriumEnergyReaction2018}.
    If the domain $\Omega = \R^d$, the \Loj inequality holds for the SHK gradient flows of $\phiP$-divergence energy for $p\in (-\infty,\frac12] \cup [1,2]$ given that \eqref{eq:Loj-W} holds.

\begin{remark}
    This theorem showcases the strength of the SHK gradient flows.
    For dimension $d\leq 4$, the \Loj inequality holds for SHK gradient flows of all $\phiP$--divergence energy!
    For $d\geq 5$, we still have the generous interval $p\in (-\infty,1/2] \cup [1-\frac1d,\infty)$, which improves significantly from the pure Otto-Wasserstein and the pure (spherical) Hellinger geometries.
\end{remark}
A direct consequence is the following qualitative statement that applies to a large family of practical energy functionals
\begin{corollary}
    The SHK gradient flows converge exponentially globally for the following energy functionals:
    KL divergence ($p=1$) under LSI, forward KL divergence ($p=0$) unconditionally, 
    $\chi^2$-divergence ($p=2$) under a \Loj inequality, reverse $\chi^2$-divergence ($p=-1$) unconditionally,
    and the Hellinger distance ($p=1/2$).
\end{corollary}

\section{Hellinger-Kantorovich gradient flows of positive measures $\Mplus$}
\label{sec:HKGF}
Unlike the spherical counterpart, the HK gradient flows over positive measures $\Mplus$ are more challenging to treat.
This is due to the absence of the global LSI type inequalities for the Otto-Wasserstein flows over positive measures $\Mplus$,
which we discuss in Section~\ref{sec:HKGF-Mplus}.
Subsequently, we provide the analysis for the HK gradient flow over positive measures $\Mplus$.
Concretely, we establish global convergence results for the $\phiP$-divergence energy for $p\in (-\infty,1/2]$, as well as
for the KL divergence energy ($p=1$)using a novel analysis via a shape-mass decomposition.

\subsection{The loss of LSI on positive measures $\Mplus$ and a sufficient condition}
\label{sec:HKGF-Mplus}
For the Wasserstein distance, the McCann condition (see, e.g., \citep{ambrosio2008gradient}) shows that 
$ \rmD_{\phiP}(\cdot |\dd x)$ (\ie the reference measure is Lebesgue) is geodesically convex only for 
$p\geq (d{-}1)/d$ where $d$ is the dimension. In \citep{LiMiSa23FPGG}, necessary and sufficient 
conditions for the geodesic convexity of entropy functionals with 
respect to the HK distance
were derived. The upper threshold $p= 1/2$ was also observed in the sense
that densities with  $ p\in [p_*,1/2]\cup (1,\infty)$ lead to geodesically convex 
$p$-divergences, where $p_*=1/3$ for space dimension $d=1$ and $p_*=1/2$ for $d=2$. For $d\geq 3$ only the range $p > 1$ is admitted. However, only the convexity constant
$\Lambda=0$ has been shown for all $p > 1$ (\ie not strongly convex). 

To improve on the state-of-the-art analysis, we first provide our result on the HK \Loj in the following corollary.
The \Loj inequality in the HK geometry over positive measures $\Mplus$ reads,
for $\alpha, \beta >0$,
\begin{align}
    \label{eq:L-WFR}
    \int 
    \left(
        \alpha
        \bigg|\nabla \dFdmu\bigg|^2
        +
        \beta
        \bigg|\dFdmu\bigg|^2
        \right) \dd \mu
    \geq 
    \beta c_* \left(
        F(\mu) - \inf_{\nu\in \Mplus} F(\nu)
    \right)
    .
\end{align}

\begin{corollary}
[A sufficient condition for HK flow]
\label{cor:loj-power-wfr}
For ${\phiP}$-divergence energy $F(\mu) = \mathrm{D}_{\phiP}(\mu |\pi)$ with
$p\in (-\infty,\frac12]$, the \Loj inequality~\eqref{eq:L-WFR} holds globally
over positive measures $ \Mplus$ with the constant
$\displaystyle c_* = \frac1{1-p}$.
\end{corollary}
In relating those results to previous geodesic convexity results for the HK
gradient flows in Table~\ref{tab:power-entropies-convexity}, we first note that
geodesic convexity implies \Loj inequality but only with a non-negative
constant $c\geq 0$.  As the dimension increases, \citet{LiMiSa23FPGG}'s result
and the McCann condition have an increasing power threshold for the value of
$p$. For dimension $d\geq 3$, their intervals no longer overlap with the
threshold of $p\leq \frac12$ for the global \Loj in the Hellinger geometry.
Yet,
we are able to
provide a further \Loj result that is weaker than \citep{LiMiSa23FPGG}'s
geodesic convexity condition;
see
Table~\ref{tab:power-entropies-convexity}.
In
previous works such as \citep{liu_polyak-l_2023},
it has been suggested that the \Loj inequality for the HK geometry holds whenever the \Loj inequalities for the Otto-Wasserstein (LSI) and Hellinger both hold.
However, we next show that such a strategy cannot result in a global \Loj inequality.

First,
if $p\leq \frac12$,
Corollary~\ref{cor:loj-power-wfr} has established the \Loj inequality \eqref{eq:Loj-FR} with a constant $c\geq \frac{1}{1-p}$.
    This directly results in the \Loj inequality in the HK geometry.
    If $p>\frac12$,
    different from the pure Fisher-Rao case,
    it does not automatically imply the absence of the HK \Loj inequality.
    This can be seen by first assuming a \Loj condition for the Otto-Wasserstein dissipation \eqref{eq:Loj-W} with a constant condition $c_{W}>0$.
    Then, since the Hellinger dissipation quantity is always non-negative along gradient flows,
    \begin{align*}
            \int 
            \mu \cdot 
            \left(
                \alpha\bigg|\nabla \dFdmu\bigg|^2
                +
                \beta
                \bigg|\dFdmu\bigg|^2
                \right)
            \geq 
            \alpha c_W\cdot \left(
                F(\mu) - \inf_{\nu\in \Mplus} F(\nu)
            \right)
            + \beta \cdot 0 
        ,
    \end{align*}
    which yields the HK \Loj with the constant $c_{W}$.
Now, it may seem that the \Loj inequality in the HK geometry can be established in this manner.
However, the situation is more nuanced
due to the Otto-Wasserstein dissipation over positive measures $\Mplus$, instead of the probability measure space $\calP$.
In such cases, \Loj inequality \emph{cannot} hold globally for the Otto-Wasserstein flow.
\begin{proposition}
    [No \Loj for Otto-Wasserstein flows over $\Mplus$]
    Given the the $\varphi$-divergence
    energy functional $\mathrm{D}_{\varphi}(\cdot |\pi)$.
    Then, there exists no global \Loj inequality~\eqref{eq:Loj-W}
    for the Otto-Wasserstein gradient flow over the positive measures $\Mplus$.
    \label{prop:no-lsi-w2}
\end{proposition}
\begin{proof}
    [Proof of proposition~\ref{prop:no-lsi-w2}]
    For any non-negative target measure $\pi\in\Mplus$,
    we pick the measure $\mu$ to be a scalar multiple of $\pi$,
    \ie $\mu = Z\,\pi$ with $Z>0$.
    See the illustration in Figure~\ref{fig:NO-LHS-WGF-Prob-meas}.
    The Radon-Nikodym derivative is a constant $\frac{\dd\mu}{\dd \pi}\equiv Z$.
    Then, the \Loj inequality reads
    $$
    0 = \biggl\|\nabla \left(\varphi^\prime \Big(\frac{\dd\mu}{\dd \pi}\Big) 
         \right)\biggr\|^2_{L^2_{\mu}}
    \geq 
    c\cdot \operatorname{\mathrm{D}_\varphi}(\mu|\pi) 
    \textrm{ for some } c>0,
    $$
    which cannot hold whenever $\operatorname{\mathrm{D}_\varphi}(\mu|\pi) =
    \varphi(Z)\pi(\Omega) > 0$.
\end{proof}
The intuition for the above proposition is that the Otto-Wasserstein flow of $\mathrm{D}_\varphi(\cdot |\pi )$ 
``gets stuck''
when the density ratio is constant $\frac{\dd\mu}{\dd \pi}\equiv Z$
since the {metric slope} is zero.
See also the illustration in Figure~\ref{fig:NO-LHS-WGF-Prob-meas}.
This result shows that we cannot hope to rely on the Otto-Wasserstein dissipation to establish the \Loj inequality in the HK flow of positive measures.
\begin{figure}[htb]
    \centering
    \includegraphics[width=0.7\linewidth]{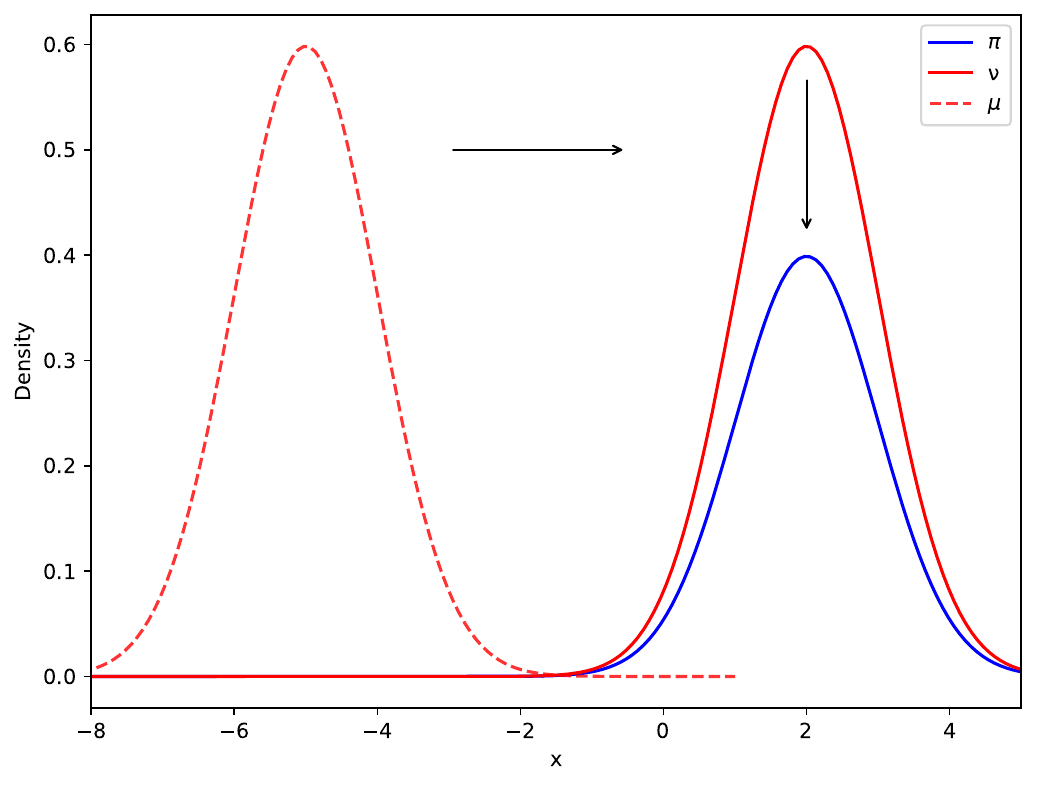}
    \caption{
        See Proposition~\ref{prop:no-lsi-w2} for the details of the functional inequality and Otto-Wasserstein flow of positive measures.
        In this plot, the density ratio $\frac{\dd\nu}{\dd \pi}$ is a constant $Z>0$.
        Hence, there is no ``Otto-Wasserstein gradient'' to drive the curve from $\nu$ towards $\pi$.
        In the opposite regime, the density ratio $\frac{\dd\mu}{\dd \pi}$ has many close-to-zero locations. Hence, there is not enough ``Hellinger gradient''
        to drive the curve from $\mu$ towards $\pi$.
    }
    \label{fig:NO-LHS-WGF-Prob-meas}
\end{figure}

\subsection{A special case: HK gradient flows of KL divergence energy}
\label{sec:HK-KL}

In contrast to the global convergence results for the SHK gradient flows in Section~\ref{sec:SHK},
results such as Proposition~\ref{prop:no-lsi-w2} might hint a pessimism about the HK gradient flows.
However, in this section, we show that the HK gradient flows of KL divergence (i.e. $\varphi_p$ for $p=1$) driving energy have a special property that still guarantees the global convergence.

To further understand the idea behind Proposition~\ref{prop:no-lsi-w2},
we first show a straightforward extension of the \eqref{eq:LSI} over probability measures to positive measures.
Without loss of generality, we assume that the target measure $\pi$ is a probability measure, \ie $\pi(\Omega)=1$.
\begin{proposition}
    [Generalized log-Sobolev inequality on $\Mplus$]
    Suppose the logarithmic Sobolev inequality~\eqref{eq:LSI} holds
    with a positive constant $c_{\textrm{LSI-}\mathcal{P}}>0$
    when restricted to probability measures (i.e. $\mu$ and $\pi$ are probability measures).
    Then, the following inequality holds for the Otto-Wasserstein gradient flow over the positive measures $\Mplus$:
    \begin{align}
        \int \bigg|{\nabla \log \dMudPi}\bigg|^2\dd \mu 
        \geq
        c_{\textrm{LSI-}\mathcal{P}}\cdot \biggl(
                    \mathrm{D}_\mathrm{KL}
                    \left(\mu|\pi\right)
                    - 
                    \left(
                        z\log z -z +1
                    \right)
                \biggr)
                ,
        \tag{LSI-$\Mplus$}
        \label{eq:LSI-Mplus}
    \end{align}
where $z:=\mu(\Omega)$ is the total mass of the measure $\mu$.
    Moreover,
we have
\begin{align}
    \int \bigg|{\nabla \log \dMudPi}\bigg|^2\dd \mu
    \geq 
    c_{\textrm{LSI-}\mathcal{P}}\cdot 
    \mathrm{D}_\mathrm{KL}\left(\mu| z\cdot \pi\right)
    .
    \label{eq:scaled-kl-wgf-lyapunov}
\end{align}
\label{prop:gen-lsi-w2-mass-preserving}
\end{proposition}
The intuition here is that the Otto-Wasserstein gradient flow, viewed as a
mass-preserving flow with total mass $\mu(\Omega)$, satisfies the LSI type
inequality.  This is illustrated in Figure~\ref{fig:NO-LHS-WGF-Prob-meas}.
\bigskip

\noindent
\begin{proof}
    [Proof of Proposition~\ref{prop:gen-lsi-w2-mass-preserving}]
    We have the logarithmic Sobolev inequality~\eqref{eq:LSI} for the probability measures $\tilde \mu:=\frac{1}{z}\cdot \mu $ where $z:=\mu (\Omega)$ is the mass of $\mu$,
    \begin{align*}
        \int \frac{\dd\tilde  \mu}{\dd \pi} \left({\nabla \log \frac{\dd \tilde \mu}{\dd \pi}}\right)^2\dd \pi 
        \geq c_{\textrm{LSI-}\mathcal{P}}\cdot \mathrm{D}_\mathrm{KL}(\tilde \mu|\pi).
    \end{align*}
    Expanding the KL divergence, we have
    \begin{multline*}
        \mathrm{D}_\mathrm{KL}(\tilde \mu|\pi)
        = \int \frac{\dd \tilde \mu}{\dd \pi} \log \frac{\dd \tilde \mu}{\dd \pi} \dd \pi 
        =\int \frac1{z} \frac{\dd  \mu}{\dd \pi}
        \left(
            \log \frac{\dd  \mu}{\dd \pi} - \log z 
        \right)
            \dd \pi
            \\ 
        = 
        \frac1{z} \left(
                \int \frac{\dd  \mu}{\dd \pi} \log \frac{\dd  \mu}{\dd \pi} \dd \pi
                -z + 1
            \right)
        - \log z +1 - \frac1{z}
        \\
        =
        \frac1{z} \mathrm{D}_\mathrm{KL}(\mu|\pi)-\frac1{z}(z\log z -z +1)
        .
    \end{multline*}
    Combining the above relation with property of the Sobolev norm, 
    \begin{align*}
        \int \dMudPi \left({\nabla \log \dMudPi}\right)^2\dd \pi 
        =
        z\cdot 
        \int \frac{\dd \tilde \mu }{\dd \pi } \left({\nabla \log \frac{\dd \tilde \mu }{\dd \pi }}\right)^2\dd \pi 
        \geq 
        c_{\textrm{LSI-}\mathcal{P}}\cdot \left(
                    \mathrm{D}_\mathrm{KL}\left(\mu|\pi\right)
                    - 
                    \left(
                        z\log z -z +1
                    \right)
                \right)
                \\
        =
        c_{\textrm{LSI-}\mathcal{P}}\cdot \left(
                    \mathrm{D}_\mathrm{KL}\left(\mu|\pi\right) - 
                    \mathrm{D}_\mathrm{KL}\left(\mu(\Omega )\cdot \pi |\pi\right) 
                \right)
                .
    \end{align*}
    For the last part of the result,
    we rewrite
    the right-hand side of the inequality above using
    the relation
    $z \log z -z +1
    =
    \mathrm{D}_\mathrm{KL}\left(\mu(\Omega )\cdot \pi |\pi\right) $.
Recall a generalized Pythagorean inequality for the KL divergence that reads
    $$\mathrm{D}_\mathrm{KL}\left(\mu|\pi\right)
    \geq
    \mathrm{D}_\mathrm{KL}\left(\mu|\pi^*\right)
    + 
    \mathrm{D}_\mathrm{KL}\left(\pi^* |\pi\right)
    ,
    $$
    where $\pi^*$ is the information projection of $\pi$ onto the positive measures of total mass $z$
    $$
    \pi^*\in \arginf \left\{
        \mathrm{D}_\mathrm{KL}\left(\gamma|\pi\right)
        \mid
        \gamma\in \Mplus, \gamma(\Omega)=z
        \right\}
        .
    $$
    By Jensen's inequality,
    \begin{align*}
        \mathrm{D}_\mathrm{KL}\left(\gamma|\pi\right)
        =\int \varphi_\mathrm{KL}\dd \pi \geq \varphi_\mathrm{KL}\left(\int \dd \pi\right)
        =\varphi_\mathrm{KL}(z),
    \end{align*}
    where the inequality holds when
    $\pi^*=z\cdot \pi$.
    Therefore,
    \begin{align}
        \mathrm{D}_\mathrm{KL}\left(\mu|\pi\right)
        \geq
        \mathrm{D}_\mathrm{KL}\left(\mu|z\cdot \pi\right)
        + 
        \mathrm{D}_\mathrm{KL}\left(z\cdot \pi |\pi\right)
        .
        \label{eq:pythagorean-kl}
    \end{align}
    Combining the results above,
    we obtain the desired inequality
    \begin{align}
        \int \dMudPi ({\nabla \log \dMudPi})^2\dd \pi 
        \geq
        c_{\textrm{LSI-}\mathcal{P}}\cdot \left(
                    \mathrm{D}_\mathrm{KL}\left(\mu|\pi\right) -
                    \mathrm{D}_\mathrm{KL}\left(\mu(\Omega )\cdot \pi |\pi\right)
                \right)
        \geq 
        C\cdot 
        \mathrm{D}_\mathrm{KL}\left(\mu|\mu(\Omega)\cdot \pi\right)
        .
    \end{align}
Thus, Proposition \ref{prop:gen-lsi-w2-mass-preserving} is established. 
\end{proof}
The insight from Proposition~\ref{prop:gen-lsi-w2-mass-preserving} also provides us
an exponentially decaying Lyapunov functional
along the Otto-Wasserstein flow over the
$\Mplus$.
Noting the property of the KL divergence
$\mathrm{D}_\mathrm{KL}\left(\mu| z\cdot \pi\right) 
=z\cdot \mathrm{D}_\mathrm{KL}\left(\frac{1}{z}\cdot \mu| \pi\right) 
$, we find
\begin{multline}
    -\frac{\dd }{\dd t}
    \mathrm{D}_\mathrm{KL}\left(\mu|z\cdot \pi\right)
    =
    \frac1z\int \dMudPi \cdot \left({\nabla\left( \log \dMudPi - \log z\right)}\right)^2\dd (z\cdot \pi )
    \\
    =
    \int \dMudPi \cdot \left({\nabla\log \dMudPi }\right)^2\dd \pi
    \overset{\text{Prop.~\ref{prop:gen-lsi-w2-mass-preserving}}}{\geq }
    c_{\textrm{LSI-}\mathcal{P}}\cdot
    \mathrm{D}_\mathrm{KL}\left(\mu| z\cdot \pi\right) 
\end{multline}
Then, by Gr\"onwall's lemma, the Lyapunov functional $\mathrm{D}_\mathrm{KL}\left(\mu| z\cdot \pi\right) $ decays exponentially along the mass-preserving Otto-Wasserstein gradient flow.
\begin{corollary}
    [$\mathrm{D}_\mathrm{KL}\left(\mu| \mu(\Omega)\cdot \pi\right)$ is Lyapunov for Otto-Wasserstein-$\Mplus$]
    For the 
    mass-preserving
    Otto-Wasserstein gradient flow over the positive measures $\Mplus$ with the KL divergence energy $F(\mu ) = \mathrm{D}_{\mathrm{KL}}(\mu  |\pi)$,
    the Lyapunov functional $\mathrm{D}_\mathrm{KL}\left(\mu| z \cdot \pi\right)$ (where $z:=\mu(\Omega)$ is the total mass of the measure $\mu$) decays exponentially along the flow, \ie
    \begin{align*}
        \mathrm{D}_{\mathrm{KL}}(\mu(t)|z\cdot \pi) \leq 
        \ee^{-c_{\textrm{LSI-}\mathcal{P}}\cdot  t} \mathrm{D}_{\mathrm{KL}}(\mu(0)|z\cdot \pi)
\text{ and }
    \mathrm{D}_{\mathrm{KL}}(\frac1{z}\mu(t)|\pi) \leq
     \ee^{-c_{\textrm{LSI-}\mathcal{P}}\cdot  t} \mathrm{D}_{\mathrm{KL}}(\frac1{z}\mu(0)|\pi)
\end{align*}
for $t\geq 0$ and the LSI constant $c_{\textrm{LSI-}\mathcal{P}}$ as in Proposition~\ref{prop:gen-lsi-w2-mass-preserving}.
    \label{cor:lyapunov-w2}
\end{corollary}
This result, combined with the Pythagorean type relation~\eqref{eq:pythagorean-kl},
shows 
the Lyapunov functional $\mathrm{D}_\mathrm{KL}\left(\mu| z \cdot \pi\right)$
decays 
towards the lower bound $\mathrm{D}_{\mathrm{KL}}(z\cdot \pi|\pi) = z\log z -z +1$.
Furthermore,
it implies that, while $\mu$ itself does not converge to $\pi$ due to Proposition~\ref{prop:no-lsi-w2},
the \emph{shape} $\frac1{\mu(\Omega)}\mu$ does converge to the target.
Using a similar idea,
we next analyze the Hellinger (He) and the spherical Hellinger (SHe) geometries.

We exploit a special property of the KL divergence, namely,
for
any constant $Z\in \R^+$,
the SHe flow
of the KL divergence energy $\mathrm{D}_{\mathrm{KL}}(\cdot | Z \pi )$
is independent of $Z$.
Yet, in the He flow of positive measures, the constant $Z$ in the minimization has an impact.
This idea can be easily seen by
calculating the gradient flow equation of the He flow
\begin{align*}
    \dot{\rho} = -\rho \log\left(\ddfrac{\rho}{\left(Z\pi\right)}\right)
    = - \rho \left(\log\left(\ddfrac{\rho}{\pi}\right) -   \log Z\right),
\end{align*}
\ie the growth field is indeed affected by the scalar $Z$.
In comparison, the scalar $Z$ is canceled for the SHe flow of probability measures
\begin{align*}
    \dot{\rho} = -\rho \left(\log\left(\ddfrac{\rho}{\left(Z\pi\right)}\right)
    - \int \rho \log\left(\ddfrac{\rho}{\left(Z\pi\right)}\right)
    \right)
    = - \rho \left(\log\left(\ddfrac{\rho}{\pi}\right) 
    - \int \rho \log\left(\ddfrac{\rho}{\pi}\right)
    \right).
\end{align*}
Since the Otto-Wasserstein flow is always mass-conserving, this difference in He and SHe is the key for our analysis next, which we term the \emph{shape-mass analysis}.
\subsection{Shape-mass analysis: global KL decay of HK gradient flows}
\label{sec:shape-mass-ana}
Our starting point is to carefully compare the HK and SHK gradient flows.
For the convenience, we remember below the 
associated gradient-flow equations of HK and SHK flows under the KL energy
\begin{align}
\label{eq:I.HK.PDE}
\tag{HK-KL}
\dot \mu &= \alpha \DIV\!\big( \nabla \mu + \ddfrac{\mu}{\pi}\nabla\pi\big) 
- \beta \mu \log\left(\ddfrac{\mu}\pi\right), 
\\
\label{eq:I.SHK.PDE}
\tag{SHK-KL}
\dot \rho &= \alpha \DIV\!\big( \nabla \rho + \ddfrac{\rho}{\pi}\nabla\pi\big) 
 - \beta \rho\,\Big( \!\log\left(\ddfrac{\rho}\pi\right)
 -\int_{\R^d}\rho\log\left(\ddfrac{\rho}\pi\right) \dd x  \Big).
\end{align}
For the clarity of the analysis, we use the symbol $\mu$ for the positive measure in the HK flow and $\rho$ for the probability measure in the SHK flow.
We exploit the following simple observation of those two equations.
\begin{theorem}
    [Relation between solutions to HK and SHK equations]
     If $t \mapsto \mu(t)$ solves \eqref{eq:I.HK.PDE}, then $t\mapsto \rho(t)=
      \frac1{z(t)} \mu(t)$ with $z(t)=\int_{\R^d} \mu(t,x)\dd x$ solves
      \eqref{eq:I.SHK.PDE}. Moreover, if $t\mapsto \rho(t)$ solves
      \eqref{eq:I.SHK.PDE}, then $t \mapsto \mu(t)=\kappa(t) \rho(t)$ solves
      \eqref{eq:I.HK.PDE} for suitable functions $t\mapsto \kappa(t)$ independent of the variable $x$.
      Furthermore,
      $\kappa (t)$ is the solution to the following equation of mass
      \begin{align}
      \dot z &= - \beta z \log z 
       - \beta z \int_{\R^d}
      \rho\log\left(\ddfrac\rho\pi\right) \dd x
      .
      \tag{Mass-HK}
      \label{eq:mass-hk}
      \end{align}
      \label{prop:relation-scaled-hk-shk}
\end{theorem}
\begin{proof}
    [Proof of Theorem~\ref{prop:relation-scaled-hk-shk}]
    {The first part of the proposition is straightforward.}
    To derive the mass equation,
    suppose $\mu$ is a solution to the HK equation~\eqref{eq:I.HK.PDE}.
    Applying the chain rule to the time derivative
    $\dot \mu = \dot z \rho + z \dot \rho$,
    where the shape-mass decomposition $\mu = z \rho$ is used.
    Plug this into the HK equation~\eqref{eq:I.HK.PDE},
    \begin{align*}
        \dot z \rho
        =
        \alpha \DIV\!\left( \nabla \mu + \ddfrac{\mu}{\pi}\nabla\pi\right) 
- \beta \mu \log\left(\ddfrac{\mu}\pi\right)
- z\dot \rho.
    \end{align*}
Since the shape $\rho$ is a probability measure, we have $\int \rho \dd x =1,\int \dot \rho \dd x =0$.
Then, we integrate both sides of the above equation to obtain
\begin{align*}
    \dot z
    =
    \alpha \int  \DIV\!\left( \nabla \mu + \ddfrac{\mu}{\pi}\nabla\pi\right)
    -\beta \int \mu \log\left(\ddfrac{\mu}\pi\right) \dd x
    \overset{\text{IBP}}{=}
    -
    \beta z\log z
    - \beta \int \rho \log\left(\ddfrac{\rho}\pi\right) \dd x
    ,
\end{align*}
which is the desired mass equation~\eqref{eq:mass-hk}.
\end{proof}

This observation reveals the key to the following shape-mass analysis we will present:
consider a general solution $t\mapsto \mu(t)$ of \eqref{eq:I.HK.PDE} and write
it in the form $\mu(t) = z(t) \rho(t)$ with 
the normalized density
$\rho(t)\in \calP(\Omega)$ describing the shape and $z(t)>0$ the total mass.
Using this observation,
We can further extend the SHK analysis to 
general target measure $\pi \in \Mplus$. The
$\rho$-equation~\eqref{eq:I.SHK.PDE} is 
mass-preserving and
invariant under
the change of variable
from 
$\pi$ to $\gamma \pi$ with
$\gamma>0$.
In that case, one expects convergence to the steady state $\gamma_\pi \pi \in
\calP(\Omega)$, where $\gamma_\pi$ is a normalizing constant.
Hence, we now denote the shape-mass decomposition of the target $\pi = z_* \pi_*$
where $\pi_*\in \calP(\Omega)$.
Then,
when
starting from a solution $t \mapsto \rho(t)$ of 
the mass-preserving flow
\eqref{eq:I.SHK.PDE} and
assuming $\pi_* \in \calP(\Omega)$, $z_0>0$, and
$z_*>0$, the mass equation \eqref{eq:mass-hk} reads
\[
\dot z = \beta\, z\,  \big( \log z_*-\rmD_{\textrm{KL}}(\rho|\pi) - \log
z \big), \qquad z(0)=z_0.
\]
Then $t\mapsto \mu(t) = z(t) \rho(t)$ is a solution of \eqref{eq:I.HK.PDE}
with the initial condition $\rho(0)= z_0 \rho(0)$ for the energy functional
$\rmD_{\textrm{KL}}(\,\cdot\,| z_*\pi_*)$.

To provide a general decay estimate for solutions of
$(\calM^+(\Omega),\calH_\rmB,\HK_{\alpha,\beta})$, we now use the 
shape-mass decomposition
$\mu(t,x)=z(t) \rho(t,x)$.
As shown in Proposition~\ref{prop:no-lsi-w2} and Proposition~\ref{prop:gen-lsi-w2-mass-preserving},
LSI cannot hold globally over $\Mplus$.
Therefore, we use the standard
log-Sobolev inequality but restricted to
the probability measures,
which is the same as in
\eqref{eq:LSI-2} and recalled here for convenience:
    for $\pi_*= \gamma_\pi \,\pi\in \calP(\Omega)$,
    \begin{equation}
      \label{eq:LSI.ass}
      \exists\, c_\text{LSI} >0\ \forall \, \rho \in \calP(\Omega):\quad 
    \int_{\Omega} \rho \big| \nabla \log(\rho/\pi_*) \big|^2 \dd x  \geq c_\text{LSI}
    \rmD_{\textrm{KL}}(\rho|\pi_*) .
    \tag{LSI-$\mathcal{P}$}
    \end{equation}
In the following result, we can see two contributions to the convergence of
$\mu(t)=z(t)\rho(t)$ to $\pi=z_*\pi_*$,
where $z_*:=\pi (\Omega)$ is the total mass of the target measure and $\pi_*$ is a probability measure, a.k.a. the shape.
We now detail the results of the shape-mass analysis for the HK-KL gradient flow.

We first provide the convergence of the mass variable $z(t)$ to the target mass $z_*$.
\begin{proposition}
    [Solution of the mass equation]
    The equation of mass~\eqref{eq:mass-hk}
    admits the explicit solution
      \begin{align}
        \label{eq:mass-sol-hk}
         z(t) = z_* \left(\frac{z_0}{z_*}\right)^{\ee^{-\beta t}} \, \ee^{-h(t)}. 
      \end{align}
      where $h(t)=\int_0^t \ee^{-\beta(t-s)} \rmD_{\textrm{KL}}(\rho(s)|\pi_*)\dd
    s$ is an auxiliary function.

    If $\rmD_{\textrm{KL}}(\rho(s)|\pi_*)\to 0$ for $t\to \infty$, then $h(t)\to 0$ and
    $z(t) \to z_*$.
    \label{prop:mass-convergence}
\end{proposition}

Setting
$H_0=\rmD_{\textrm{KL}}(\rho(0)|\pi_*)$ and $\wh\alpha = \alpha c_{\text{\rm LSI-}\mathcal{P}}$,
we now deliver the convergence of the shape $\rho(t)$ to the target shape $\pi_*$ and the mass $z(t)$ to the target mass $z_*$.
\begin{proposition}
    [Shape and mass convergence]
The normalized probability measure $\rho(t) = \frac1{z(t)}\mu(t)$ (the shape) converges to the target $\pi_*$ exponentially in KL divergence along the HK gradient flow,
\ie
\begin{align*}
\rmD_{\textrm{KL}}(\rho(t)|\pi_*) &\leq \ee^{-\wh\alpha t}H_0.
\tag{shape convergence}
\end{align*}
The mass variable $z(t)$ converges to the target mass $z_*$ exponentially,
\ie
\begin{align*}
|z(t) - z_*| &\leq \max\{z_0,z_*\} \left|\log\left(\frac{z_0}{z_*}\right)\right| \ee^{-\beta t} + H_0 \frac{\ee^{-\wh\alpha t} - \ee^{-\beta t}}{\beta - \wh\alpha}.
\tag{mass convergence}
\end{align*}
\label{prop:DecayShape}
\end{proposition}
Note that the convergence rate of the shape
$\rho(t)$ to the limiting shape $\pi_*$ is dominated by the transport part alone,
with an exponential decay rate $\wh\alpha = \alpha C_\text{LSI}$. The total mass
can only be changed by the growth through the Hellinger dissipation. Hence, the decay
rate is simply $\beta$, but it may be delayed by $\ee^{-\wh\alpha t}$ if the
shape converges only slowly.

Combining the results of Proposition~\ref{prop:mass-convergence} and Proposition~\ref{prop:DecayShape}, we can now provide the global exponential decay analysis for the HK-KL gradient flow in the sense of the Hellinger distance.

\begin{theorem}[Convergence to equilibrium via shape-mass analysis]
\label{th:DecayEquil}
The following convergence estimate in the Hellinger distance holds
\begin{equation}
\label{eq:Decay22}
\He(\mu(t),\pi)
\leq
\left( \frac{\max\{z_0^{1/2},z_*^{1/2}\}}2 \,H_0^{1/2} + 
z_*^{1/2} \left(  g\left(\left( \frac{z_0}{z_*}\right)^{1/2}\right) + \frac1{\wh\alpha} \right)
 \right)     \ee^{-\gamma t}  \quad \text{for } t>0
 ,
\end{equation}
where $\gamma= \min\big\{ \beta, \wh\alpha/2 \big\}$ and $g(a) = 
\max\{ \log (1/a), a{-}1\}\geq 0 $. \EEE
\end{theorem}
Before delving into the proof, we highlight that the singularity of $g(a)=\log(1/a)$ 
(for $a<1$) is needed to cover the case that, for a very small initial mass $z_0$, it takes a 
long time to build up enough mass to see the exponential decay to the limiting profile.

The above results imply that we cannot have a  global \Loj inequality for the HK gradient 
flow over the positive measures $\Mplus$.
However, the last theorem shows that, for 
the KL divergence as driving energy, global exponential decay is still guaranteed.
An exception is the case that we start with $\mu=0$,
which remains an unstable steady state of the flow.
\begin{proof}[Proof of Theorem \ref{th:DecayEquil}]
We use the shape-mass decomposition $\mu(t)=z(t)\rho(t)$ and $\pi = z_*\pi_*$ with $\pi_*,
\rho(t) \in \calP(\R^d)$.
We first estimate the convergence 
of $\rho$ to $\pi_*$ via
\begin{multline*}
    -\frac\rmd{\rmd t} \rmD_{\textrm{KL}}(\rho|\pi_*) = \int_{\R^d} \Big( \alpha \rho\!\;
    \big|\nabla\log(\rho/\pi_*)\big|^2 + 
    \beta \rho \big(\log(\rho/\pi_*) - \int \rho 
    \log (\rho/\pi_*) 
    \big)^2
    \Big) \dd x
    \\
    \geq \wh\alpha \rmD_{\textrm{KL}}(\rho|\pi_*) + \beta \cdot 0,
\end{multline*} 
where we ignored the term due to the 
spherical Hellinger geometry since it's non-negative.
Thus, we obtain 
\[
\rmD_{\textrm{KL}}(\rho(t)|\pi_*) \leq \ee^{-\wh\alpha t} H_0 \quad \text{with }
\wh\alpha=  \alpha c_{\text{LSI-}\calP} \text{ and } H_0=\rmD_{\textrm{KL}}(\rho(0)|\pi_*).
\]

Next, we use the relation for $z(t)$ with $z_0=\int_{\R^d} \mu(0,x) \dd x$: 
\[
   z(t) = z_* \big(z_0/z_*\big)^{\ee^{-\beta t}} \, \ee^{-h(t)} \text{ with } 
h(t)= \int_0^t \ee^{-\beta(t-s)} \rmD_{\textrm{KL}}(\rho(s)|\pi_*) \dd s . 
\]
Using the estimate for $\rmD_{\textrm{KL}}(\rho(t)|\pi_*)$ we have 
$0\leq h(t) \leq H(t):= \EEE H_0 \big(\ee^{-\wh\alpha t} {-} \ee^{-\beta t} \big) /
(\beta{-}\wh\alpha)$. Moreover, for all $a ,t>0$ we have 
\[
\big|a^{\ee^{-\beta t}} -1 \big| \leq g(a) 
\,\ee^{-\beta t} \quad \text{where } g(a) 
= \sup\nolimits_{x\in(0,1)}\tfrac{|a^x{-}1|}{x} = \max\{ \log(1/a), a{-}1\}.
\]
Using $\ee^{-h(t)}\leq 1$, we find, for $\sigma\in {]0,1]}$, the estimate 
\begin{align*}
|z(t)^\sigma{-}z_*^\sigma|\ &\leq \  \big|z_*^\sigma (z_0/z_*)^{\sigma\,\ee^{-\beta
    t}} \ee^{-h(t)\sigma}  -  z_*^\sigma \ee^{-h(t)\sigma}\big| \,+\,
\big| z_*^\sigma \ee^{-h(t)\sigma} -z_*^\sigma\big|   
\\
&\leq \ z_*^\sigma \big|(z_0/z_*)^{\sigma\,\ee^{-\beta t}} {-} 1 \big|
  + z_*^\sigma \sigma H(t) 
\ \leq \ 
\ z_*^\sigma \, g\big(z_0^\sigma/z_*^\sigma\big) \, 
\ee^{-\beta t} + z_*^\sigma \sigma H(t) .
\end{align*}
We estimated the last term on the first line by $|\ee^{-x}{-}1| \leq x$ 
for all $x>0$, using $x=\sigma h(t)$.

For the full estimate, we use the classical bound
$4\He(\rho,\pi)^2 = 
2\calD_{\phi_{1/2}}(\rho|\pi) \leq \rmD_{\textrm{KL}}(\rho|\pi) \leq H_0 \ee^{-\wh\alpha t}$.  With
$z(t)\leq \max\{z_0, z_*\}$, we are now able to establish 
\eqref{eq:Decay22} as follows:
\begin{align*}
\He(\mu,\pi)&= \He(z(t)\rho,z_*\pi_*)\leq \He(z(t)\rho,z(t)\pi_*)+ 
  \He(z(t)\pi_*,z_*\pi_*)\\[0.2em]
& =\sqrt{z} \,\He(\rho,\pi) + |\sqrt{z}{-}\sqrt{z_*}| 
\\
&
\leq \max\{z_0^{1/2},z_*^{1/2}\}\frac{H_0^{1/2}}2 \ee^{\wh\alpha t/2} 
  + z_*^{1/2} g\big( (z_0/z_*)^{1/2} \big) \ee^{-\beta t}  
  + \frac{z_*^{1/2}}2 H(t).
\end{align*}
Moreover, we establish $H(t)\leq (2/\wh\alpha) \;\!\ee^{-\gamma t}$ with $\gamma= \min\{\beta,\wh\alpha/2\}$ as follows:
For  $\beta\leq \wh\alpha/2$, we have
\[
\frac{\ee^{-\wh\alpha t} {-} \ee^{-\beta t} }{ \beta{-}\wh\alpha} 
  = \ee^{-\beta t} \int_0^t  \ee^{(\beta-\wh\alpha) s} \dd s
\leq \ee^{-\beta t} \int_0^t  \ee^{-\wh\alpha s/2}\dd s \leq
\frac2{\wh\alpha}\,\ee^{-\beta t}
\]
and for $\beta \geq \wh\alpha/2$ we estimate as follows:
\[
\frac{\ee^{-\wh\alpha t} {-} \ee^{-\beta t} }{ \beta{-}\wh\alpha} 
  =  \int_0^t \ee^{-\beta(t-s)} \ee^{\wh\alpha s} \dd s
\leq  \int_0^t \ee^{-\wh\alpha(t-s)/2} \ee^{-\wh\alpha s}\dd s \leq
\frac2{\wh\alpha}\,\ee^{-\wh\alpha t/2}. 
\]
Putting together the results, the desired estimate \eqref{eq:Decay22} is established. 
\end{proof}
A decay estimate for 
$ \mathrm{D_{KL}}(\mu(t)|\pi)$ similar to \eqref{eq:Decay22} can also be derived
by using the decomposition  $\mathrm{D_{KL}}(\mu(t)|\pi) 
= \mathrm{D_{KL}}(z(t)\rho(t) |z_*\pi_*) = z(t)\;\mathrm{D_{KL} } (\rho(t)|\pi_*) + z_* \lambda(z(t)/z_*)$ with $\lambda(r)=r \log r - r+1$.
We omit the elementary proof to avoid redundancy.

\begin{remark}
[{Beyond the KL energy functional}]
It is tempting to generalize the above analysis to general $\varphi$-divergence energy $F(\mu) = \mathrm{D}_{\varphi}(\mu |\pi)$, beyond the KL case.
However, we now present the following observation that such a generalization is difficult.

Using a similar 
shape-mass decomposition $\mu = z \rho$
as in the proof of Theorem~\ref{prop:relation-scaled-hk-shk},
we extract the equation
\begin{align}
    \label{eq:shape-eq}
    \dot \rho
    =
    \alpha \DIV\!\left( 
        \rho \nabla \dFdmu
        \right) 
    - \beta \rho  \left(
                    \dFdmu
                    +
                    \frac{\dot z}{z} 
                    \right)
                    .
\end{align}
We integrate both sides and again use the fact that $\rho$ remains a probability measure along the mass-preserving flow.
Noting the simple relation
$\dFdmu 
=\frac{ \delta F}{\delta \mu }[z\rho]$,
we obtain
\begin{align}
    \frac{\dot z}{z} 
    =
    -\int \rho\cdot 
    \frac{ \delta F}{\delta \mu }[z\rho]
    .
    \label{eq:mass-eq-zdot-over-z}
\end{align}
Therefore, the shape equation~\eqref{eq:shape-eq} can be rewritten as
\begin{align*}
    \dot \rho
    =
    \alpha \DIV\!\left( 
        {\rho} \nabla\frac{ \delta F}{\delta \mu  }[z\rho]
        \right) 
    - {\beta \rho} 
    \left(
        \frac{ \delta F}{\delta \mu  }[z\rho]
                    -\int \rho\cdot 
    \frac{ \delta F}{\delta \mu  }[z\rho]
                    \right)
\end{align*}
Specialized to the $\varphi_p$-divergence energy $F(\mu) = \mathrm{D}_{\varphi_p}(\mu |\pi)$, we have
\begin{align}
    \label{eq:shape-eq-shin}
     \rho
    =
    \alpha \DIV\!\left( 
        \rho  \nabla \varphi_p^\prime\left(z\frac{\dd\rho }{\dd \pi}\right)
        \right) 
        -
        {\beta \rho} 
    \left(
        \varphi_p^\prime\left(z\frac{\dd\rho }{\dd \pi}\right)
                    -\int \rho\cdot 
    \varphi_p^\prime\left(z\frac{\dd\rho }{\dd \pi}\right)
                    \right)
    \end{align}
This shape equation~\eqref{eq:shape-eq-shin} reveals the insight about the SHK flow of the $\varphi_p$-divergence energy.
If $p=1$, \ie the KL divergence energy, the shape equation~\eqref{eq:shape-eq-shin} simplifies to the energy equation of \eqref{eq:I.SHK.PDE}, which is the observation of Theorem~\ref{prop:relation-scaled-hk-shk}
and Section~\ref{sec:HK-KL}.

In the case of $p\neq 1$, the shape equation~\eqref{eq:shape-eq-shin} is \emph{not} the SHK gradient flow equation by itself -- the shape and mass variables are coupled.
Therefore, the observation of Theorem~\ref{prop:relation-scaled-hk-shk} does not hold for other $\varphi_p$-divergence energies than the KL.
For example, in the case of $p=2$, the shape equation~\eqref{eq:shape-eq-shin} reads
\begin{align*}
    \dot \rho
    =
    \alpha \cdot z \DIV\!\left( 
        \rho \nabla \frac{\dd\rho }{\dd \pi}
        \right)
    - \beta\cdot  z \rho  \left(
                   \frac{\dd\rho }{\dd \pi}
                    -
                    \int \rho\cdot 
                   \frac{\dd\rho }{\dd \pi}
                    \right)
\end{align*}
where the right-hand side has a coupled mass variable $z$.
Hence, it is not the SHK gradient flow equation.
In this sense, our shap-mass analysis is specifically designed for the HK-KL gradient flow.
\end{remark}

\appendix
\section{Further proofs and technical results}
\label{sec:proof}

We need the following technical properties to prepare for the characterization of the Lyapunov functional of the HK gradient flows.
\begin{lemma}
    \label{lm:m-pq-relations}
    The constant $m_{p,q}$ satisfies the following estimates:
\\
(a) $m_{p,q}>0$ if and only if $p \leq \wh p(q)$ with 
$\wh p(q):= \max\big\{ 0, \min\{ 1,1{-}q\} \big\}$,
\\
(b) For $p\leq 1/2$ we have $m_{p,p} = 1/(1{-}p)$. 
\\
(c) For $p\in [0,1]$ we have $m_{p,1-p}= \min\{1/p,1/(1{-}p)\} \in [1,2]$.  
\end{lemma}
\begin{proof}
We define $N_{p,q}(r)= r \DphiP(r)\varphi'_q(r)/\phi_q(r)$ which can be
continuously extended at $r=1$ by the value $N_{p,q}(1)= 2$. Thus, using the
continuity and  positivity of 
$N_{p,q}:(0,\infty) \to (0,\infty)$, we obtain $m_{p,q}>0$ if and only if the
two limits for $N_{p,q}(0)=\lim_{r\to 0}N_{p,q}(r)$ and
$N_{p,q}(\infty)=\lim_{r\to \infty}N_{p,q}(\infty)$ are positive as well. 

The asymptotic behavior for $r\to \infty$ is easily discussed:
\[
 N_{p,q}(\infty) = \begin{cases} \infty& \text{for }p\geq 1 , \\ 
 \frac{\max\{q,1\}}{1{-}p} &\text{for } p<1. 
\end{cases}
\]
To see this, we first observe that $r \varphi'_q(r)/\phi_q(r) \to \max\{1,q\}$ for
$r\to \infty$. Second, we have $\DphiP(r) \to \infty$ for $\geq 1$ and
$\DphiP(r) \to 1/(1{-} p)$ for $p<1$. 

The determination of $N_{p,q}(0)$ needs a more detailed case-by-case study, but
is elementary. We obtain 
\[
N_{p,q}(0)= \begin{cases}
 0 & \text{for }q\geq 1 \text{ and } p\geq \max\{0,1{-}q\}, \\
\frac{q}{q{-} 1} & \text{for }q>1 \text{ and } p=0, \\
\frac{1}{1{-} q} & \text{for }q\in {]0,1]} \text{ and } p=1{-}q, \\
\frac{-q}{p{-}1} & \text{for } q<0 \text{ and } p>0,\\
\infty & \text{otherwise}.  
\end{cases}
\]
With this, part (a) is established. 

To see part (b) we observe $\frac\rmd{\rmd r} N_{p,p}(r) \leq0$ and find
$m_{p,p}=N_{p,p}(\infty)=1/(1{-}p)$.

Similarly, for part (c) one shows $\frac\rmd{\rmd r} N_{p,1-p}(r) \leq 0$ for
$p\in [0,1/2]$ giving $m_{p,1-p}=N_{p,1-p}(\infty)=1/(1{-}p)$. Moreover, for
$p\in [1/2,1]$ one shows $\frac\rmd{\rmd r} N_{p,1-p}(r) \geq 0$, which implies
$m_{p,1-p}= N_{p,1-p}(0)= 1/p$.   
\end{proof}

\begin{proof}
    [Proof of Corollary~\ref{cor:loj-power}]
According to the previous result, we need to find $c_*=c_p$ which is given via
\[
\frac1{c_p}= \sup_{0<s\neq 1} \Phi(s) \quad \text{with } \Phi(s) := \EEE 
\frac{\phiP(s)}{s (\varphi'_p(s))^2} .
\]
Observe that $\varphi_{1/2}(s)=2(\sqrt{s}-1)^2$ implies
$\Phi_{1/2} \equiv 1/2$, and hence $c_{1/2}=2$.

The derivative of the power-like entropy generator~\eqref{eq:power-ent} is
\begin{align*}
    \varphi'_p(s) = \frac{s^{p-1} - 1}{p-1} \textrm{ for }
    p\in
    \R\setminus\{0,1\},
    \quad
    \varphi_0'(s)=1-\frac1s,
    \quad
    \varphi_1'(s) =\log s
    .
\end{align*}
For general $p\in \R$,
an explicit calculation yields  
\begin{align*}
    \Phi(s)=\frac{p-1}{p}\cdot 
    \frac{s^p-ps + p-1}{s (s^{p-1}-1)^2}
    ,
\end{align*}
we easily verify that $\Phi$ is continuous at the $s=1$ and hence continuous on
$(0,\infty)$.  Moreover, we have $\Phi(s) \to \max\{0,1{-}p\}$ for
$s\to \infty$. For $s \to 0$ we obtain $\Phi(s) \to \infty$ for $p>1/2$ and
$\Phi(s)\to 0$ for $p<1/2$.

Thus, we conclude $\sup \Phi=\infty$ for $p>1/2$. For $p\leq 1/2$ a closer inspection shows that $\sup \Phi=1{-}p$. and hence $c_p=1/(1{-}p)$ as stated. 
\end{proof}

\begin{proof}
    [Proof of Corollary~\ref{cor:loj-power-wfr}]
By our threshold condition,
for $p\in (-\infty,\frac12]$,
the constant $c_{\He}=\frac1{1-p}$
satisfies
\begin{align*}
    \bigg\|\dFdmu\bigg\|^2_{L^2_{\mu}}
    \geq 
    c_{\He}\cdot \left(
        F(\mu) - F(\pi)
    \right)
    .
\end{align*}
Since the dissipation of the Otto-Wasserstein part is always
non-negative, we trivially have
\begin{align*}
    \alpha
    \bigg\|\nabla \dFdmu\bigg\|^2_{L^2_{\mu}}
    +
    \beta
    \bigg\|\dFdmu\bigg\|^2_{L^2_{\mu}}
    \geq 
    \beta
    c_{\He}\cdot \left(
        F(\mu) - F(\pi)
    \right)
    +0
    ,
\end{align*}
which is the desired statement.
\end{proof}

\paragraph*{Acknowledgment}
This project was partially supported by
the Deutsche Forschungsgemeinschaft (DFG) through the Berlin Mathematics
Research Center MATH+ (EXC-2046/1, project ID: 390685689).

\small
\bibliography{ref}

\begin{thebibliography}{57}
\providecommand{\natexlab}[1]{#1}
\providecommand{\url}[1]{\texttt{#1}}
\expandafter\ifx\csname urlstyle\endcsname\relax
  \providecommand{\doi}[1]{doi: #1}\else
  \providecommand{\doi}{doi: \begingroup \urlstyle{rm}\Url}\fi

\bibitem[Amari(1998)]{amari1998natural}
S.-I. Amari.
\newblock Natural gradient works efficiently in learning.
\newblock \emph{Neural computation}, 10\penalty0 (2):\penalty0 251--276, 1998.

\bibitem[Amari and Nagaoka(2000)]{amari2000methods}
S.-i. Amari and H.~Nagaoka.
\newblock \emph{Methods of information geometry}, volume 191.
\newblock American Mathematical Soc., 2000.

\bibitem[Ambrosio et~al.(2005)Ambrosio, Gigli, and Savare]{ambrosio2008gradient}
L.~Ambrosio, N.~Gigli, and G.~Savare.
\newblock \emph{Gradient flows: in metric spaces and in the space of probability measures}.
\newblock Springer Science \& Business Media, 2005.

\bibitem[Arnold et~al.(2001)Arnold, Markowich, Toscani, and Unterreiter]{arnold2001convex}
A.~Arnold, P.~Markowich, G.~Toscani, and A.~Unterreiter.
\newblock On convex {Sobolev} inequalities and the rate of convergence to equilibrium for {Fokker-Planck} type equations.
\newblock \emph{Comm. Partial Diff. Eqns.}, 26\penalty0 (1-2):\penalty0 43--100, 2001.

\bibitem[Bakry and {\'E}mery(1985)]{bakryDiffusionsHypercontractives1985}
D.~Bakry and M.~{\'E}mery.
\newblock Diffusions hypercontractives.
\newblock In J.~Az{\'e}ma and M.~Yor, editors, \emph{S{\'e}minaire de {{Probabilit{\'e}s XIX}} 1983/84}, volume 1123, pages 177--206. {Springer Berlin Heidelberg}, {Berlin, Heidelberg}, 1985.
\newblock \doi{10.1007/BFb0075847}.

\bibitem[Bauer et~al.(2016)Bauer, Bruveris, and Michor]{BaBrMi16UFRM}
M.~Bauer, M.~Bruveris, and P.~W. Michor.
\newblock Uniqueness of the {Fisher-Rao} metric on the space of smooth densities.
\newblock \emph{Bull. Lond. Math. Soc.}, 48\penalty0 (3):\penalty0 499--506, 2016.

\bibitem[Beck and Teboulle(2003)]{beckMirrorDescentNonlinear2003}
A.~Beck and M.~Teboulle.
\newblock Mirror descent and nonlinear projected subgradient methods for convex optimization.
\newblock \emph{Operations Research Letters}, 31\penalty0 (3):\penalty0 167--175, May 2003.
\newblock ISSN 01676377.
\newblock \doi{10.1016/S0167-6377(02)00231-6}.

\bibitem[Bhattacharya(1942)]{Bhat42DD}
A.~Bhattacharya.
\newblock On discrimination and divergence.
\newblock \emph{Proc.\ Indian Sci. Congress}, Part III:\penalty0 13, 1942.

\bibitem[Blanchet and Bolte(2018)]{blanchetFamilyFunctionalInequalities2018}
A.~Blanchet and J.~Bolte.
\newblock A family of functional inequalities: {{{\L}ojasiewicz}} inequalities and displacement convex functions.
\newblock \emph{Journal of Functional Analysis}, 275\penalty0 (7):\penalty0 1650--1673, Oct. 2018.
\newblock ISSN 0022-1236.
\newblock \doi{10.1016/j.jfa.2018.06.014}.

\bibitem[Blei et~al.(2017)Blei, Kucukelbir, and McAuliffe]{blei_variational_2017}
D.~M. Blei, A.~Kucukelbir, and J.~D. McAuliffe.
\newblock Variational inference: A review for statisticians.
\newblock \emph{Journal of the American Statistical Association}, 112\penalty0 (518):\penalty0 859--877, Apr. 2017.
\newblock ISSN 0162-1459, 1537-274X.
\newblock \doi{10.1080/01621459.2017.1285773}.

\bibitem[Carrillo et~al.(2024)Carrillo, Chen, Huang, Huang, and Wei]{carrilloFisherRaoGradientFlow2024}
J.~A. Carrillo, Y.~Chen, D.~Z. Huang, J.~Huang, and D.~Wei.
\newblock Fisher-{{Rao}} gradient flow: Geodesic convexity and functional inequalities.
\newblock \penalty0 (arXiv:2407.15693), July 2024.
\newblock \doi{10.48550/arXiv.2407.15693}.

\bibitem[Chewi et~al.(2020)Chewi, Gouic, Lu, Maunu, and Rigollet]{chewiSVGDKernelizedWasserstein2020}
S.~Chewi, T.~L. Gouic, C.~Lu, T.~Maunu, and P.~Rigollet.
\newblock {{SVGD}} as a kernelized {{Wasserstein}} gradient flow of the {Chi}-squared divergence.
\newblock \emph{preprint}, arXiv:2006.02509, June 2020.

\bibitem[Chizat(2022)]{chizat2022sparse}
L.~Chizat.
\newblock Sparse optimization on measures with over-parameterized gradient descent.
\newblock \emph{Mathematical Programming}, 194\penalty0 (1-2):\penalty0 487--532, 2022.

\bibitem[Chizat et~al.(2018{\natexlab{a}})Chizat, Peyr{\'e}, Schmitzer, and Vialard]{chizatInterpolatingDistanceOptimal2018}
L.~Chizat, G.~Peyr{\'e}, B.~Schmitzer, and F.-X. Vialard.
\newblock An interpolating distance between optimal transport and {{Fisher--Rao}} metrics.
\newblock \emph{Foundations of Computational Mathematics}, 18\penalty0 (1):\penalty0 1--44, Feb. 2018{\natexlab{a}}.
\newblock ISSN 1615-3375, 1615-3383.
\newblock \doi{10.1007/s10208-016-9331-y}.

\bibitem[Chizat et~al.(2018{\natexlab{b}})Chizat, Peyré, Schmitzer, and Vialard]{chizat_unbalanced_2019}
L.~Chizat, G.~Peyré, B.~Schmitzer, and F.-X. Vialard.
\newblock Unbalanced optimal transport: Dynamic and {{Kantorovich}} formulation.
\newblock \emph{J. Functional Analysis}, 274\penalty0 (11):\penalty0 3090--3123, 2018{\natexlab{b}}.

\bibitem[Csisz{\'a}r(1967)]{csiszar1967information}
I.~Csisz{\'a}r.
\newblock On information-type measure of difference of probability distributions and indirect observations.
\newblock \emph{Studia Sci. Math. Hungar.}, 2:\penalty0 299--318, 1967.

\bibitem[Dai et~al.(2016)Dai, He, Dai, and Song]{daiProvableBayesianInference2016}
B.~Dai, N.~He, H.~Dai, and L.~Song.
\newblock Provable {{Bayesian}} inference via particle mirror descent.
\newblock In \emph{Proceedings of the 19th {{International Conference}} on {{Artificial Intelligence}} and {{Statistics}}}, pages 985--994. PMLR, May 2016.

\bibitem[{Domingo-Enrich} and Pooladian(2023)]{domingo-enrichExplicitExpansionKullbackLeibler2023}
C.~{Domingo-Enrich} and A.-A. Pooladian.
\newblock An explicit expansion of the {{Kullback-Leibler}} divergence along its {{Fisher-Rao}} gradient flow.
\newblock \emph{Preprint}, arXiv:2302.12229, Feb. 2023.

\bibitem[Gallou{\"e}t and Monsaingeon(2017)]{gallouet2017jko}
T.~O. Gallou{\"e}t and L.~Monsaingeon.
\newblock A {{JKO}} splitting scheme for {{Kantorovich--Fisher--Rao}} gradient flows.
\newblock \emph{SIAM Journal on Mathematical Analysis}, 49\penalty0 (2):\penalty0 1100--1130, 2017.

\bibitem[Gladin et~al.(2024)Gladin, Dvurechensky, Mielke, and Zhu]{gladin2024interaction}
E.~Gladin, P.~Dvurechensky, A.~Mielke, and J.-J. Zhu.
\newblock Interaction-force transport gradient flows.
\newblock \emph{arXiv preprint arXiv:2405.17075}, 2024.

\bibitem[Hellinger(1909)]{hellingerNeueBegrundungTheorie1909}
E.~Hellinger.
\newblock Neue {{Begr{\"u}ndung}} der {{Theorie}} quadratischer {{Formen}} von unendlich vielen {{Ver{\"a}nderlichen}}.
\newblock \emph{J. reine angew. Mathematik}, 1909\penalty0 (136):\penalty0 210--271, 1909.

\bibitem[Hoffman et~al.(2013)Hoffman, Blei, Wang, and Paisley]{hoffman2013stochastic}
M.~D. Hoffman, D.~M. Blei, C.~Wang, and J.~Paisley.
\newblock Stochastic variational inference.
\newblock \emph{Journal of Machine Learning Research}, 2013.

\bibitem[Jordan et~al.(1999)Jordan, Ghahramani, Jaakkola, and Saul]{jordanIntroductionVariationalMethods1999}
M.~I. Jordan, Z.~Ghahramani, T.~S. Jaakkola, and L.~K. Saul.
\newblock An introduction to variational methods for graphical models.
\newblock \emph{Machine Learning}, 37\penalty0 (2):\penalty0 183--233, Nov. 1999.
\newblock ISSN 1573-0565.
\newblock \doi{10.1023/A:1007665907178}.

\bibitem[Jordan et~al.(1998)Jordan, Kinderlehrer, and Otto]{jordan_variational_1998}
R.~Jordan, D.~Kinderlehrer, and F.~Otto.
\newblock The variational formulation of the {Fokker}–{Planck} equation.
\newblock \emph{SIAM J. Mathematical Analysis}, 29\penalty0 (1):\penalty0 1--17, 1998.

\bibitem[Kakutani(1948)]{kakutaniEquivalenceInfiniteProduct1948}
S.~Kakutani.
\newblock On equivalence of infinite product measures.
\newblock \emph{Annals of Mathematics}, 49\penalty0 (1):\penalty0 214--224, 1948.

\bibitem[Karimi et~al.(2020)Karimi, Nutini, and Schmidt]{karimiLinearConvergenceGradient2020}
H.~Karimi, J.~Nutini, and M.~Schmidt.
\newblock Linear convergence of gradient and proximal-gradient methods under the {Polyak-Lojasiewicz} condition.
\newblock \emph{Preprint}, arXiv:1608.04636, Sept. 2020.

\bibitem[Khan and Nielsen(2018)]{khan2018fast}
M.~E. Khan and D.~Nielsen.
\newblock Fast yet simple natural-gradient descent for variational inference in complex models.
\newblock In \emph{2018 International Symposium on Information Theory and Its Applications (ISITA)}, pages 31--35. IEEE, 2018.

\bibitem[Khan and Rue(2023)]{khanBayesianLearningRule2023}
M.~E. Khan and H.~Rue.
\newblock The {{Bayesian}} learning rule.
\newblock \emph{Preprint}, arXiv:2107.04562, June 2023.

\bibitem[Kondratyev et~al.(2016)Kondratyev, Monsaingeon, and Vorotnikov]{kondratyevNewOptimalTransport2016}
S.~Kondratyev, L.~Monsaingeon, and D.~Vorotnikov.
\newblock A new optimal transport distance on the space of finite {Radon} measures.
\newblock \emph{Adv. Differ. Eqns.}, 21\penalty0 (11/12):\penalty0 1117--1164, 2016.

\bibitem[Krichene et~al.(2015)Krichene, Bayen, and Bartlett]{kricheneAcceleratedMirrorDescent2015}
W.~Krichene, A.~Bayen, and P.~L. Bartlett.
\newblock Accelerated mirror descent in continuous and discrete time.
\newblock In \emph{Advances in {{Neural Information Processing Systems}}}, volume~28. Curran Associates, Inc., 2015.

\bibitem[Lambert et~al.(2022)Lambert, Chewi, Bach, Bonnabel, and Rigollet]{lambertVariationalInferenceWasserstein2022}
M.~Lambert, S.~Chewi, F.~Bach, S.~Bonnabel, and P.~Rigollet.
\newblock Variational inference via {{Wasserstein}} gradient flows.
\newblock \emph{Preprint}, arXiv:2205.15902, Oct. 2022.

\bibitem[Laschos and Mielke(2019)]{LasMie19GPCA}
V.~Laschos and A.~Mielke.
\newblock Geometric properties of cones with applications on the {H}ellinger--{K}antorovich space, and a new distance on the space of probability measures.
\newblock \emph{J. Funct. Analysis}, 276\penalty0 (11):\penalty0 3529--3576, 2019.
\newblock \doi{10.1016/j.jfa.2018.12.013}.

\bibitem[Laschos and Mielke(2023)]{LasMie23EVIH}
V.~Laschos and A.~Mielke.
\newblock Evolutionary variational inequalites on the {Hellinger-Kantorovich} and the spherical {Hellinger-Kantorovich} spaces.
\newblock \emph{Preprint}, arXiv:2207.09815v3, 2023.

\bibitem[Liero et~al.(2016)Liero, Mielke, and Savar\'e]{LiMiSa16OTCR}
M.~Liero, A.~Mielke, and G.~Savar\'e.
\newblock Optimal transport in competition with reaction -- the {H}ellinger--{K}antorovich distance and geodesic curves.
\newblock \emph{SIAM J. Math. Anal.}, 48\penalty0 (4):\penalty0 2869--2911, 2016.

\bibitem[Liero et~al.(2018)Liero, Mielke, and Savaré]{liero_optimal_2018}
M.~Liero, A.~Mielke, and G.~Savaré.
\newblock Optimal entropy-transport problems and a new {Hellinger}–{Kantorovich} distance between positive measures.
\newblock \emph{Inventiones mathematicae}, 211\penalty0 (3):\penalty0 969--1117, 2018.
\newblock \doi{10.1007/s00222-017-0759-8}.

\bibitem[Liero et~al.(2023)Liero, Mielke, and Savar\'e]{LiMiSa23FPGG}
M.~Liero, A.~Mielke, and G.~Savar\'e.
\newblock Fine properties of geodesics and geodesic $\lambda$-convexity for the {Hellinger--Kantorovich} distance.
\newblock \emph{Arch. Rat. Mech. Analysis}, 247\penalty0 (112):\penalty0 1--73, 2023.
\newblock \doi{10.1007/s00205-023-01941-1}.

\bibitem[Liu et~al.(2023)Liu, Majka, and Szpruch]{liu_polyak-l_2023}
L.~Liu, M.~B. Majka, and L.~Szpruch.
\newblock Polyak-{L}ojasiewicz inequality on the space of measures and convergence of mean-field birth-death processes.
\newblock \emph{Applied Mathematics \& Optimization}, 87\penalty0 (3):\penalty0 48, June 2023.
\newblock \doi{10.1007/s00245-022-09962-0}.

\bibitem[Lu et~al.(2023)Lu, Slep{\v{c}}ev, and Wang]{lu2023birth}
Y.~Lu, D.~Slep{\v{c}}ev, and L.~Wang.
\newblock Birth--death dynamics for sampling: global convergence, approximations and their asymptotics.
\newblock \emph{Nonlinearity}, 36\penalty0 (11):\penalty0 5731, 2023.

\bibitem[Mielke(2023)]{mielke2023introduction}
A.~Mielke.
\newblock An introduction to the analysis of gradients systems.
\newblock \emph{arXiv preprint arXiv:2306.05026}, 2023.

\bibitem[Mielke(2024)]{Miel24EVIF}
A.~Mielke.
\newblock On {EVI} flows in the (spherical) {Hellinger-Kantorovich} space.
\newblock \emph{Oberwolfach Reports}, 21\penalty0 (1):\penalty0 313--318, 2024.
\newblock \doi{10.4171/OWR/2024/7}.

\bibitem[Mielke and Mittnenzweig(2018)]{mielkeConvergenceEquilibriumEnergyReaction2018}
A.~Mielke and M.~Mittnenzweig.
\newblock Convergence to equilibrium in energy-reaction--diffusion systems using vector-valued functional inequalities.
\newblock \emph{J. Nonlinear Science}, 28\penalty0 (2):\penalty0 765--806, 2018.

\bibitem[Nemirovskij and Yudin(1983)]{nemirovskijProblemComplexityMethod1983}
A.~S. Nemirovskij and D.~B. Yudin.
\newblock \emph{Problem Complexity and Method Efficiency in Optimization}.
\newblock Wiley-Interscience, 1983.

\bibitem[Ohta and Takatsu(2011)]{ohtaDisplacementConvexityGeneralized2011}
S.-i. Ohta and A.~Takatsu.
\newblock Displacement convexity of generalized relative entropies.
\newblock \emph{Advances in Mathematics}, 228\penalty0 (3):\penalty0 1742--1787, Oct. 2011.
\newblock \doi{10.1016/j.aim.2011.06.029}.

\bibitem[Otto(1996)]{otto1996double}
F.~Otto.
\newblock Double degenerate diffusion equations as steepest descent.
\newblock Preprint no.\ 480, SFB 256, University of Bonn, 1996.

\bibitem[Otto(2001)]{ottoGeometryDissipativeEvolution2001}
F.~Otto.
\newblock The geometry of dissipative evolution equations: The porous medium equation.
\newblock \emph{Comm. Partial Diff. Eqns.}, 26\penalty0 (1-2):\penalty0 101--174, 2001.

\bibitem[Otto and Villani(2000)]{otto2000generalization}
F.~Otto and C.~Villani.
\newblock Generalization of an inequality by {Talagrand} and links with the logarithmic {Sobolev} inequality.
\newblock \emph{Journal of Functional Analysis}, 173\penalty0 (2):\penalty0 361--400, 2000.

\bibitem[Peletier(2014)]{peletier_variational_2014}
M.~A. Peletier.
\newblock Variational {Modelling}: {Energies}, gradient flows, and large deviations.
\newblock \emph{arXiv:1402.1990 [math-ph]}, Feb. 2014.
\newblock arXiv: 1402.1990.

\bibitem[Rao(1945{\natexlab{a}})]{Rao45IAAE}
C.~R. Rao.
\newblock Information and the accuracy attainable in the estimation of statistical parameters.
\newblock \emph{Bull. Calcutta Math. Soc.}, 37:\penalty0 81--91, 1945{\natexlab{a}}.

\bibitem[Rao(1945{\natexlab{b}})]{rao1945information}
C.~R. Rao.
\newblock Information and accuracy attainable in the estimation of statistical parameters.
\newblock \emph{Bulletin of the Calcutta Mathematical Society}, 37\penalty0 (3):\penalty0 81--91, 1945{\natexlab{b}}.

\bibitem[Rayleigh(1873)]{rayleigh_general_1873}
J.~W. S.~B. Rayleigh.
\newblock \emph{Some general theorems relating to vibrations}.
\newblock London Mathematical Society, 1873.

\bibitem[Rotskoff et~al.(2019)Rotskoff, Jelassi, Bruna, and Vanden-Eijnden]{rotskoff2019neuron}
G.~Rotskoff, S.~Jelassi, J.~Bruna, and E.~Vanden-Eijnden.
\newblock Neuron birth-death dynamics accelerates gradient descent and converges asymptotically.
\newblock In \emph{International conference on machine learning}, pages 5508--5517. PMLR, 2019.

\bibitem[Santambrogio(2015)]{santambrogio_optimal_2015}
F.~Santambrogio.
\newblock Optimal transport for applied mathematicians.
\newblock \emph{Birkäuser, NY}, 55\penalty0 (58-63):\penalty0 94, 2015.
\newblock Publisher: Springer.

\bibitem[Santambrogio(2017)]{santambrogioEuclideanMetricWasserstein2017}
F.~Santambrogio.
\newblock \{\vphantom\}{{Euclidean}}, metric, and {{Wasserstein}}\vphantom\{\} gradient flows: An overview.
\newblock \emph{Bulletin of Mathematical Sciences}, 7\penalty0 (1):\penalty0 87--154, 2017.
\newblock \doi{10.1007/s13373-017-0101-1}.

\bibitem[Wainwright and Jordan(2008)]{wainwrightGraphicalModelsExponential2008}
M.~J. Wainwright and M.~I. Jordan.
\newblock Graphical models, exponential families, and variational inference.
\newblock \emph{Foundations and Trends$^{\text{\rm\textregistered}}$ in Machine Learning}, 1\penalty0 (1{\textendash}2):\penalty0 1--305, 2008.

\bibitem[Wibisono et~al.(2016)Wibisono, Wilson, and Jordan]{wibisonoVariationalPerspectiveAccelerated2016}
A.~Wibisono, A.~C. Wilson, and M.~I. Jordan.
\newblock A variational perspective on accelerated methods in optimization.
\newblock \emph{Proceedings of the National Academy of Sciences}, 113\penalty0 (47), Nov. 2016.
\newblock ISSN 0027-8424, 1091-6490.
\newblock \doi{10.1073/pnas.1614734113}.

\bibitem[Yan et~al.(2024)Yan, Wang, and Rigollet]{yanLearningGaussianMixtures2023}
Y.~Yan, K.~Wang, and P.~Rigollet.
\newblock Learning {Gaussian} mixtures using the {Wasserstein-Fisher-Rao} gradient flow.
\newblock \emph{The Annals of Statistics}, 52\penalty0 (4):\penalty0 1774--1795, 2024.

\bibitem[Zhu(2024)]{zhu2024inclusive}
J.-J. Zhu.
\newblock Inclusive {KL} minimization: A {Wasserstein-Fisher-Rao} gradient flow perspective.
\newblock \emph{arXiv preprint arXiv:2411.00214}, 2024.

\end{thebibliography}

\end{document}